\documentclass[12pt, reqno]{amsart}
\usepackage{amsmath, amsthm, amscd, amsfonts, amssymb, graphicx, color, mathrsfs}
\usepackage[bookmarksnumbered, colorlinks, plainpages]{hyperref}
\usepackage{cite}
\usepackage[all]{xy}
\usepackage{slashed}
\usepackage{tikz-cd}
\usepackage{mathabx}
\usepackage{tipa}
\usepackage{soul}
\usepackage{cancel}
\usepackage{ulem}
\usepackage{algorithm}
\usepackage{algpseudocode}
\usepackage{enumitem}
\usepackage{multirow}
\usepackage{float}
\usepackage{verbatim}
\usepackage[
  left=2cm,
  right=2cm,
  top=2cm,
  bottom=2cm
]{geometry}

\DeclareMathOperator{\Ad}{Ad}
\DeclareMathOperator{\ad}{ad}
\DeclareMathOperator{\GL}{GL}

\DeclareMathOperator{\SO}{SO}

\DeclareMathOperator{\Span}{span}

\allowdisplaybreaks

\newlength\tindent
\newtheorem{theorem}{Theorem}[section]

\newtheorem{proposition}[theorem]{Proposition}

\theoremstyle{definition}

\theoremstyle{remark}
\newtheorem{remark}[theorem]{Remark}
\numberwithin{equation}{section}

\begin{document}
\setcounter{page}{1}

\title[Einstein Metrics on the Complex Projective Space]{Cohomogeneity One Einstein Metrics on Complex Projective Spaces}

\author[A. Araujo]{Anderson L. A. de Araujo}
\address{Anderson L. A. de Araujo \endgraf 
Universidade Federal de Vi\c{c}osa, Departamento de Matem\'atica\\ Avenida Peter Henry Rolfs, s/n\\ CEP 36570-900, Vi\c{c}osa, MG, Brazil \endgraf
{\it E-mail address:} {\rm anderson.araujo@ufv.br}
 }
 
\author[B. Grajales]{Brian Grajales}
\address{Brian Grajales \endgraf
Universidade Estadual de Maringá, Departamento de Matem\'{a}tica, Avenida Colombo, 5790, Campus Universitário, 87020-900, Maringá - PR, Brazil 
\endgraf
  {\it E-mail address:} {\rm bdgtriana@uem.br}
  }

\author[L. Grama]{Lino Grama}
\address{Lino Grama \endgraf
IMECC-Unicamp, Departamento de Matem\'{a}tica. Rua S\'{e}rgio Buarque de Holanda \endgraf 651, Cidade Universit\'{a}ria Zeferino Vaz. 13083-859, Campinas - SP, Brazil
\endgraf
{\it E-mail address:} {\rm lino@ime.unicamp.br}
  }


\begin{abstract}
We study Einstein metrics on complex projective spaces that are invariant under cohomogeneity one actions of compact connected Lie groups, under the assumption that the singular orbits are totally geodesic. These actions were classified by Takagi into five models. For each of them, we write the Einstein equation for diagonal invariant metrics and determine the corresponding smoothness conditions at the singular orbits. Our main result is the nonexistence of smooth globally defined diagonal invariant Einstein metrics with totally geodesic singular orbits in four of the five models and a necessary condition for global existence in the remaining one.
\end{abstract}

\maketitle

\tableofcontents
\allowdisplaybreaks

\section{Introduction}
 
In this paper, we study Einstein metrics on complex projective spaces admitting a cohomogeneity one action by a compact connected Lie group. In this setting, the Einstein equation \( \operatorname{Ric}(g) = \lambda g \) reduces to a system of nonlinear ordinary differential equations along a one-dimensional orbit space. \\

In the homogeneous setting, Einstein metrics on complex projective spaces are well understood. In addition to the standard Fubini-Study metric, $\mathbb{CP}^{2n-1}$ admits a non-K\"ahler homogeneous Einstein metric induced by the transitive action of $\mathrm{Sp}(n)$. This metric is nearly K\"ahler and arises from the twistor fibration over the quaternionic projective space $\mathbb{HP}^{n-1}$ (see, e.g., Ziller~\cite{Ziller1982}). More generally, the classification of homogeneous Einstein metrics on compact homogeneous spaces has been extensively studied; we refer to~\cite{Besse1987, LauretWill2025Aligned, LauretWill2025HxH, Schwahn2022, SchwahnSemmelmann2025} for further examples and results.\\

In contrast, the study of Einstein metrics in the cohomogeneity one setting is significantly more subtle and remains far from complete. A fundamental breakthrough in this direction is due to B\"ohm~\cite{B}, who constructed infinitely many inhomogeneous Einstein metrics on spheres via cohomogeneity one actions. These examples highlight the richness of the cohomogeneity one framework and show that, beyond the homogeneous case, the space of Einstein metrics can be considerably more intricate. Recently, Cao-Labora and Rodríguez-Vázquez~\cite{CaoLaboraRodriguezVazquez2026} constructed the first inhomogeneous Einstein metrics on complex projective spaces. Their examples exist on $\mathbb{CP}^n$ for $3\leq n\leq7$ and are invariant under the cohomogeneity one action of $\mathrm{SO}(n+1)$ corresponding to Model {\bf C} below. In particular, these metrics provide new global solutions of the Einstein equation within one of the cohomogeneity one actions considered in this paper.\\

Motivated by these developments, we return to complex projective spaces and consider $\mathbb{CP}^n$ endowed with a cohomogeneity one action. Our approach relies on an explicit description of the Ricci tensor in this setting, leading to a system of nonlinear ordinary differential equations whose detailed analysis yields new results on the existence of invariant Einstein metrics on $\mathbb{CP}^n$. \\

We restrict our attention to the case in which the singular orbits are totally geodesic. Even under this assumption, the problem remains highly nontrivial since the Einstein equations lead to a singular boundary value problem near the orbits that requires delicate analysis. This condition is automatically satisfied, for example, for the cohomogeneity one action of $\mathrm{U}(p+1)\times \mathrm{U}(q+1)$ on $\mathbb{CP}^{p+q+1}$.\\ 
    
We briefly recall the geometric structure behind the problem. Let $G$ be a compact connected Lie group acting smoothly on a manifold $M$ with cohomogeneity one, and assume that $M/G\simeq[0,1].$ Then, the boundary points correspond to the singular orbits and the interior points correspond to the principal orbits of the action. For a $G$-invariant metric $g,$ let $c(t),\ 0\leq t\leq1,$ be a geodesic orthogonal to all $G$-orbits.  Up to homothety, we may assume that $c$ is parametrized by arc length. Then, the restriction of $g$ to the union of the principal orbits can be written as
\[ 
g=dt^2+g(t),
\]
where $(g(t))_{t\in(0,1)}$ is a one-parameter family of $G$-invariant metrics on a principal orbit. The family $(g(t))$ is determined by finitely many smooth functions defined in $(0,1).$ In this setting, a {\it globally defined} smooth invariant metric means that these functions extend to $[0,1]$ and satisfy certain smoothness conditions at the boundary points.\\

Cohomogeneity one actions on complex projective spaces are well understood. By \cite[Theorem 3.1]{PT}, every cohomogeneity one action of a compact connected Lie group on $\mathbb{CP}^n$ is orbit equivalent to an action induced by a Hermitian symmetric pair. These actions were classified by Takagi in \cite{T}. For the convenience of the reader, we briefly recall this classification below and refer to Section~\ref{sec:cohomo-one} for a more detailed discussion.

More precisely, if $G$ acts on $\mathbb{CP}^n$ with cohomogeneity one, then the action is orbit equivalent to one associated with a Hermitian symmetric pair $(\widetilde{U},\widetilde{K})$. The corresponding Lie algebras are listed in the following table.

\begin{table}[H]
    \centering
    \begin{tabular}{|c|c|c|}
        \hline & & \\ \textnormal{\textbf{Model}} & $\mathfrak{\tilde{u}}$ & $\mathfrak{\tilde{k}}$\\\hline
        \multirow{2}{*}{\textnormal{\textbf{A}}} & $\mathfrak{su}(p+1)\oplus\mathfrak{su}(q+1),$ &\multirow{2}{*}{$\mathfrak{s}(\mathfrak{u}(p)\oplus\mathfrak{u}(1))\oplus\mathfrak{s}(\mathfrak{u}(q)\oplus\mathfrak{u}(1))$}\\
        & $ p\geq q\geq 1,\ p>1$ &\\\hline
        \multirow{2}{*}{\textnormal{\textbf{B}}} & $\mathfrak{su}(n+3),$ &\multirow{2}{*}{$\mathfrak{s}(\mathfrak{u}(n+1)\oplus\mathfrak{u}(2))$}\\
        & $ n\geq 2$ &\\\hline
        \multirow{2}{*}{\textnormal{\textbf{C}}} & $\mathfrak{so}(n+2),$ &\multirow{2}{*}{$\mathfrak{so}(n)\oplus\mathbb{R}$}\\
        &$ n\geq 3$ &\\\hline
        \multirow{1}{*}{\textnormal{\textbf{D}}} & $\mathfrak{so}(10)$ &\multirow{1}{*}{$\mathfrak{u}(5)$}\\\hline
        \multirow{1}{*}{\textnormal{\textbf{E}}} & $\mathfrak{e}_6$ &\multirow{1}{*}{$\mathfrak{so}(10)\oplus\mathbb{R}$}\\\hline
    \end{tabular}
    \vspace{0.3cm}
    \caption{Models for cohomogeneity one actions on complex projective spaces.}
    \label{table:1}
\end{table}

In this paper, we fix a representative action for each model and describe explicitly the principal and singular isotropy groups, the singular orbits, and the corresponding decomposition of the isotropy representation. For each model, we derive the Einstein equations for diagonal invariant metrics. In contrast to the homogeneous setting, where the Einstein condition reduces to a system of algebraic equations, here it leads to a system of ordinary differential equations. The dimension of this system depends on the isotropy decomposition of the principal orbit: it is equal to \(3\) in Models {\bf A} and {\bf C}, and equal to \(5\) in Models {\bf B}, {\bf D}, and {\bf E}. One of the main difficulties in this reduction arises from the presence of equivalent isotropy summands in certain models. Although we restrict our attention to diagonal metrics, these equivalences still influence the Einstein equations. In particular, in Model {\bf B} an additional relation appears due to the non-diagonal components of the Ricci tensor.\\

Our approach is based on the work of Eschenburg and Wang~\cite{EW}. In their paper, the Einstein equation near a singular orbit is treated as a singular initial value problem, and, under suitable assumptions, local existence of invariant Einstein metrics is established for prescribed data on the singular orbit. Using this framework, we derive the systems induced by the Einstein equation in each model and determine the corresponding smoothness conditions at the singular orbits. \\

The results depend strongly on the model under consideration. In Model {\bf A}, the Einstein system subject to the smoothness conditions admits infinitely many local solutions near each singular orbit, and we obtain a necessary compatibility condition on the boundary data for the existence of a smooth globally defined invariant Einstein metric. In Models {\bf C}, {\bf D}, and {\bf E}, we construct local Einstein metrics defined on neighborhoods of the singular orbits. In Model {\bf B}, the reduced system together with the corresponding boundary conditions also admits local solutions near each singular orbit; however, because of the additional non-diagonal equation, these solutions are not automatically Einstein metrics. We further prove that there exists no smooth globally defined invariant Einstein metric of the type considered here in Models {\bf B}, {\bf C}, {\bf D}, and {\bf E}.\\

It is worth pointing out that the Fubini-Study metric is invariant under all the actions considered here and hence provides a smooth globally defined Einstein metric in every model. In Model {\bf A}, both singular orbits, $\mathbb{CP}^p$ and $\mathbb{CP}^q$, are totally geodesic, so the Fubini-Study metric belongs to the class considered here (see Remark~\ref{remark:FS:model:A}). In Models {\bf B}, {\bf D}, and {\bf E}, neither singular orbit is totally geodesic with respect to the Fubini-Study metric. Model {\bf C} lies between these two situations since the singular orbit diffeomorphic to $\mathbb{RP}^n$ is totally geodesic, whereas the other singular orbit is not. Thus, the Fubini-Study metric satisfies our assumptions only in Model {\bf A}. Our main interest, however, lies in the possible existence of inhomogeneous invariant Einstein metrics, whereas the Fubini-Study metric is homogeneous. The recent inhomogeneous Einstein metrics of Cao-Labora and Rodríguez-Vázquez~\cite{CaoLaboraRodriguezVazquez2026} in Model {\bf C} provide examples of this phenomenon.\\

The nonexistence results are obtained from an incompatibility between the solutions of the Einstein equation and the smoothness conditions at the singular orbits. More precisely, the metric is first determined by solving an initial value problem induced by the Einstein equation together with one of the boundary conditions. We then identify an invariant subspace for the resulting system. In Model {\bf C}, the local solution determined by the smoothness conditions is unique, and the invariant subspace does not satisfy the second boundary condition, yielding a contradiction. In Models {\bf B}, {\bf D}, and {\bf E}, uniqueness fails; nevertheless, we show that every local solution still lies in the same invariant subspace, whereas the smoothness conditions at the opposite singular orbit do not. This again leads to a contradiction and proves the nonexistence of smooth globally defined invariant Einstein metrics under the assumptions considered here.\\

The paper is organized as follows. In Section 2, we recall the general form of the Einstein equation for diagonal cohomogeneity one metrics and the corresponding smoothness conditions at the singular orbits. In Section 3, we describe the cohomogeneity one actions on complex projective spaces that will be considered in the paper. In Section 4, we obtain the Einstein equation for each model and study the corresponding local and global existence problems. 

\subsection{Conventions}
Throughout the paper, several symbols (such as $G$, $H$, $K_i$, $Q_i$, and $E_{ij},F_{ij}$)
are reused in different sections with different meanings; they should always be interpreted in the context
of the model under discussion. We use the symbol $\approx$ to indicate that
two manifolds are diffeomorphic. We write $\mathrm{Tr}(\cdot)$ for the matrix trace. Finally, $\overline{U}$ denotes the entrywise
complex conjugate of a matrix $U$.

\section{Einstein Metrics on Cohomogeneity One Manifolds}\label{section:2}

Let $G$ be a compact connected Lie group acting smoothly on a manifold $M$ with cohomogeneity one, assume that there are no exceptional orbits and that $M/G\approx [0,1]$. In this case, the boundary points $0$ and $1$ correspond to the singular orbits and the interior points correspond to principal orbits. Fix an arbitrary $G$-invariant Riemannian metric $\hat{g}$ on $M$. Then one can choose a geodesic $c(t),\ 0\leq t\leq 1,$ which is orthogonal to all $G$-orbits and whose projection parametrizes the quotient $M/G$. Up to homothety, we may assume that $c$ is parametrized by arc length.\\

Denote by $H$ the isotropy group of any element in a principal orbit $P=G\cdot c(t)$, $0<t<1$, and by $K_i$, $i=0,1$, the isotropy group of an element in the
singular orbit $Q_i=G\cdot c(i)$, $i=0,1$. Up to conjugacy, we may assume that $H\subset K_i$ for $i=0,1$, and we have the identifications $P\approx G/H$,
$Q_i\approx G/K_i$, and $K_i/H\approx\mathbb{S}^{\ell_i}$ for some $\ell_i\geq 1$, where $\mathbb{S}^{\ell}$ denotes the $\ell$-dimensional unit sphere (see \cite[Theorem 2]{M}). The manifold $M$ is then diffeomorphic to the union of two disk bundles over the two singular orbits whose fibers are open discs in the normal spaces.\\

The metric $\hat{g}$ is determined by its restriction $g$ to the union $M_0$ of all the principal orbits, and this restriction can be written as
\begin{equation}\label{metric:g}
g=dt^2+g(t)
\end{equation}
where $\{g(t):t\in(0,1)\}$ is a one-parameter family of $G$-invariant metrics on $P=G/H$ that can be extended smoothly to the closed interval $[0,1]$.\\ 

Let $\mathfrak{g},$ $\mathfrak{k}_0$ and $\mathfrak{h}$ denote the Lie algebras of $G,$ $K_0$ and $H,$ respectively. Consider an $\textnormal{Ad}(G)$-invariant inner product $(\cdot,\cdot)$ on $\mathfrak{g}$, and let $\mathfrak{m}:=\mathfrak{h}^\perp$ be the orthogonal complement of $\mathfrak{h}$ in $\mathfrak{g}$ with respect to this inner product. The isotropy representation of $G/H$ is equivalent to the adjoint representation $$\Ad^H\big{|}_{\mathfrak{m}}:H\longrightarrow \GL(\mathfrak{m});\ h\longmapsto \Ad(h)\big{|}_{\mathfrak{m}}.$$
We fix irreducible, pairwise $(\cdot,\cdot)$-orthogonal subrepresentations $\mathfrak{m}_1,...,\mathfrak{m}_s$ such that, for each $i\in\{1,...,s\},$ either $\mathfrak{m}_i\subseteq\mathfrak{k}_0$ or $\mathfrak{m}_i\cap\mathfrak{k}_0=\{0\}.$ Define the subspaces $$\mathfrak{p}_0:=\bigoplus\limits_{i\in J_+}\mathfrak{m}_i\ \textnormal{and}\ \mathfrak{n}_0:=\bigoplus\limits_{i\in J_-}\mathfrak{m}_i,$$ where $J_+=\{i\in\{1,...,s\}:\mathfrak{m}_i\subseteq\mathfrak{k}_0\}$ and $J_-=\{i\in\{1,...,s\}:\mathfrak{m}_i\cap\mathfrak{k}_0=\{0\}\}.$ The subspaces $\mathfrak{p}_0$ and $\mathfrak{n}_0$ can be identified with $T_{H}\left(K_0/H\right)=T_{H}\mathbb{S}^{\ell_0}$ and $T_{K_0}(G/K_0)=T_{K_0}Q_0,$ respectively. We assume that $\mathfrak{n}_0$ decomposes into pairwise inequivalent irreducible $\Ad(K_0)$-invariant subspaces as
\begin{equation}\label{decomposition:G/K}
\mathfrak{n}_0=\bigoplus_{j=1}^{N}\mathfrak{r}_j,
\end{equation}
where
\[
\mathfrak{r}_j=\bigoplus_{i\in J_-^j}\mathfrak{m}_i,\ j=1,\dots,N,\ \textnormal{and}\ J_-=J_-^1\cup\cdots\cup J_-^N.
\]
For each $t\in(0,1)$, the metric $g(t)$ can be identified with an $\Ad(H)$-invariant inner product on $\mathfrak{m}$. Consequently, there exists a positive, $(\cdot,\cdot)$-selfadjoint, $\textnormal{Ad}^H\big{|}_{\mathfrak{m}}$-equivariant operator $P(t):\mathfrak{m}\rightarrow\mathfrak{m}$ such that $$g(t)(X,Y)=(P(t)X,Y),\ X,Y\in\mathfrak{m}.$$ A diagonal metric is obtained when $$P(t)=f_1(t)^2\mathrm{Id}_{\mathfrak{m}_1}+\cdots+f_s(t)^2\mathrm{Id}_{\mathfrak{m}_s}$$
where $f_i:(0,1) \to \mathbb{R}^+,\ i=1,...,s$ are smooth functions. In this case, the shape operator $L(t)$ of the principal orbit through $c(t),\ 0<t<1$ is given by $$L(t)=\frac{1}{2}P(t)^{-1}P'(t)=\frac{f_1'(t)}{f_1(t)}\mathrm{Id}_{\mathfrak{m}_1}+\cdots+\frac{f_s'(t)}{f_s(t)}\mathrm{Id}_{\mathfrak{m}_s}.$$
If $T=\partial/\partial t,$ then the tangent space $T_{c(t)}M$ of $M$ at $c(t)$ can be decomposed as the $g$-orthogonal sum $\Span\{T\}\oplus\mathfrak{m},$ where $\mathfrak{m}$ is identified with  $T_{c(t)}\left(G\cdot c(t)\right).$ For an $(\cdot,\cdot)$-orthonormal basis $\{e_\alpha\}$ of $\mathfrak{m}$ adapted to the decomposition 
\begin{equation*}
    \mathfrak{m}=\mathfrak{m}_1\oplus\cdots\oplus\mathfrak{m}_s,
\end{equation*}
define the symbol
\begin{equation*}
\displaystyle[ijk]=\sum\limits_{e_\alpha\in\mathfrak{m}_i}\sum\limits_{e_\beta\in\mathfrak{m}_j}\sum\limits_{e_\gamma\in\mathfrak{m}_k}\left(\left[e_\alpha,e_\beta\right],e_\gamma\right)^2.
\end{equation*}
Note that $[ijk]$ is symmetric in $i,j,k.$ If $B$ denotes the Killing form of $\mathfrak{g},$ then there exist constants $b_i$ such that 
\begin{equation}\label{equation:b_i}
    -B\big{|}_{\mathfrak{m}_i\times\mathfrak{m}_i}=b_i(\cdot,\cdot)\big{|}_{\mathfrak{m}_i\times\mathfrak{m}_i},\ i=1,...,s.
\end{equation}

\begin{proposition}\label{Einstein:equations:Proposition}
    Let $\{e_\alpha\}$ as above. The metric 
    \begin{equation}\label{metric:g:2}
    g=dt^2+f_1(t)^2(\cdot,\cdot)\big{|}_{\mathfrak{m}_1\times\mathfrak{m}_1}+\cdots+f_s(t)^2(\cdot,\cdot)\big{|}_{\mathfrak{m}_s\times\mathfrak{m}_s}
    \end{equation} is Einstein for the constant $\lambda\in\mathbb{R}$ if and only if the functions $f_i,\ i=1,...,s$ satisfy the following equations:
    \begin{align}
            &d_1\frac{f_1''}{f_1}+\cdots+d_s\frac{f_s''}{f_s}=-\lambda, \label{Einstein1}\\
            &\frac{f_i''}{f_i}-\frac{b_i}{2f_i^2}-\frac{1}{d_i}\sum\limits_{k,\ell=1}^{s}\frac{f_i^4-2f_k^4}{4f_i^2f_k^2f_{\ell}^2}[ik\ell]+(d_i-1)\frac{f_i'^2}{f_i^2}+\sum\limits_{\begin{subarray}{c}
            k=1\\k\neq i    
            \end{subarray}}^{s}d_k\frac{f_i'f_k'}{f_if_k}=-\lambda,\label{Einstein2}\\
            &\sum\limits_{k,\ell=1}^{s}\frac{f_i^2f_j^2-2f_k^4+2f_k^2f_\ell^2}{4f_k^2f_\ell^2}\sum\limits_{e_\gamma\in\mathfrak{m}_k}\left(\left[e_\delta,e_\gamma\right]_{\mathfrak{m}_\ell},\left[e_\eta,e_\gamma\right]_{\mathfrak{m}_\ell}\right)=0,\ i\neq j,\label{Einstein3}
    \end{align} 
    where $e_\delta\in\mathfrak{m}_i,$ $e_\eta\in\mathfrak{m}_j$ and $d_k:=\dim\mathfrak{m}_k,\ k=1,...,s.$
\end{proposition}
\begin{proof}
It follows directly from the formulas for the Ricci tensor of a diagonal cohomogeneity one metric given in \cite[Proposition 1.14, Remark 1.16b]{GZ}. 
\end{proof}

Note that the metric \eqref{metric:g:2} is undefined at the singular orbits. However, by \cite[Corollary 2.6]{EW}, if the metric determined by \eqref{metric:g:2} admits a smooth $G$-invariant extension to the singular orbits and the functions $f_1,\ldots,f_s$ satisfy the equations \eqref{Einstein2}-\eqref{Einstein3}, then the resulting metric is Einstein on the entire manifold $M$. Following the results of Sections 3 and 5 in \cite{EW}, we will explicitly determine the values of the functions $f_1, \dots, f_s$ and their derivatives at $t=0$ to guarantee the existence of an Einstein metric on an open neighborhood of the singular orbit $Q_0.$\\
    
Let $c_i,\ i\in J_+$ be positive numbers such that $$\sum\limits_{i\in J_+}c_i^2(\cdot,\cdot)\big{|}_{\mathfrak{m}_i\times\mathfrak{m}_i}$$ 
coincides with the standard metric of curvature one on $\mathfrak{p}_0=\bigoplus\limits_{i\in J_+}\mathfrak{m}_i=T_{H}\mathbb{S}^{\ell_0}.$ With the notations used in \cite{EW}, we will decompose the metric operator $P(t)$ at $t,$ and its corresponding shape operator $L(t)$  as
\begin{equation*}
P(t)=P(t)^++P(t)^-\ \textnormal{and}\ L(t)=L(t)^++L(t)^-,   
\end{equation*}
where $P(t)^+,L(t)^+$ are $\Ad(H)$-invariant endomorphisms of $\mathfrak{p}_0$ and $P(t)^-,L(t)^-$ are $\Ad(H)$-invariant endomorphisms of $\mathfrak{n}_0.$ Hence,
\begin{align*}
   &P(t)^+=\sum\limits_{i\in J_+}f_i(t)^2\mathrm{Id}_{\mathfrak{m}_i},\ P(t)^{-}=\sum\limits_{i\in J_-}f_i(t)^2\mathrm{Id}_{\mathfrak{m}_i},\\
   &L(t)^+=\sum\limits_{i\in J_+}\frac{f_i'(t)}{f_i(t)}\mathrm{Id}_{\mathfrak{m}_i},\ L(t)^{-}=\sum\limits_{i\in J_-}\frac{f_i'(t)}{f_i(t)}\mathrm{Id}_{\mathfrak{m}_i}.
\end{align*}
 Now, according to \cite[Sections 1 and 3]{EW}, for the metric \eqref{metric:g:2} satisfying \eqref{Einstein2}-\eqref{Einstein3} to admit a smooth extension to an open neighborhood of $Q_0,$ the endomorphisms $P(t)^{\pm},$ $L(t)^{\pm}$ must satisfy the following conditions:
    \begin{equation}\label{extension:conditions}
    \begin{array}{lll}
        \displaystyle\lim\limits_{t\rightarrow 0^+}\frac{1}{t^2}P(t)^{+}=\sum\limits_{i\in J_+}c_i^2\mathrm{Id}_{\mathfrak{m}_i}, && \lim\limits_{t\rightarrow 0^+} P(t)^-=P_0,\\
        \\
        \displaystyle\lim\limits_{t\rightarrow 0^+}L(t)^{+}-\frac{1}{t}\mathrm{Id}_{\mathfrak{p}_0}=0, && \lim\limits_{t\rightarrow 0^+} L(t)^-=L_0,\\
    \end{array}
    \end{equation}
    where $P_0$ the metric operator of an arbitrary $G$-invariant metric on the singular orbit $Q_0$ and $L_0$ is the shape operator of $Q_0$ in the normal direction determined by $c$. In particular, $L_0$ is a traceless symmetric $H$-invariant operator on $\mathfrak{n}_0$. Since we are assuming $Q_0$ to be totally geodesic, we have $L_0=0.$ By \eqref{decomposition:G/K}, any such metric operator $P_0$ has the form  $P_0=\zeta_1^2\mathrm{Id}_{\mathfrak{r_1}}+\cdots+\zeta_N^2\mathrm{Id}_{\mathfrak{r}_N},$ for some arbitrary $\zeta_1,...,\zeta_N>0.$ Therefore, the conditions \eqref{extension:conditions} can be expressed in terms of the functions $f_1,...,f_s$ as 
    \begin{equation*}
        \begin{array}{lll}
        \displaystyle\lim\limits_{t\rightarrow 0^+}\frac{f_i(t)^2}{t^2}=c_i^2,\ i\in J_+, && \lim\limits_{t\rightarrow 0^+} f_i(t)=\zeta_j,\ i\in J_-^j,\\
        \\
        \displaystyle\lim\limits_{t\rightarrow 0^+}\frac{tf_i'(t)-f_i(t)}{tf_i(t)}=0,\ i\in J_+, && \lim\limits_{t\rightarrow 0^+} f_i'(t)=0,\ i\in J_-,\\
    \end{array}
    \end{equation*}
    for some arbitrary $\zeta_1,...,\zeta_N>0$ or, equivalently,
    \begin{align}\label{general:initial:conditions}
    \begin{split}
        &f_i(0)=f_i''(0)=0,\ f_i'(0)=c_i,\ i\in J_+,\\
        &f_i(0)=\zeta_j,\ i\in J_-^j,\ f_i'(0)=0,\ i\in J_-.
    \end{split}
    \end{align}
An analogous discussion applies to the singular orbit $Q_1=G/K_1$. In fact, the corresponding final smoothness conditions follow from the above ones by the reparametrization $t\mapsto 1-t,$ interchanging $K_0$ with $K_1$ and the corresponding decompositions.

\section{Cohomogeneity One Actions on Complex Projective Spaces} \label{sec:cohomo-one}
Let $(\tilde{U}, \tilde{K})$ be a Hermitian symmetric pair of rank two, where $\tilde{U}$ is a compact semisimple Lie group and $\tilde{K}$ is a closed Lie subgroup of $\tilde{U}$. Let $\mathfrak{\tilde{u}}$ and $\mathfrak{\tilde{k}}$ denote the Lie algebras of $\tilde{U}$ and $\tilde{K}$, respectively. This Hermitian symmetric pair induces a cohomogeneity one action on a complex projective space, as described below.\\

Assume that $\tilde{\mathfrak{u}} = \mathfrak{\tilde{k}} \oplus \mathfrak{\tilde{p}}$ is a Cartan decomposition. The action $\tilde{K} \times \mathfrak{\tilde{p}} \ni (\tilde{k}, X) \mapsto \Ad(\tilde{k}) X \in \mathfrak{\tilde{p}}$ has cohomogeneity one. It is known that there exists an element $Z \in \mathfrak{\tilde{k}}$ such that $J:= \ad(Z)\big{|}_{\mathfrak{\tilde{p}}}$ defines an $\Ad(\tilde{K})$-invariant complex structure on $\mathfrak{\tilde{p}}$. Consequently, we can view $\mathfrak{\tilde{p}}$ as a complex vector space isomorphic to $\mathbb{C}^{n+1}$ for some $n \in \mathbb{N}$. The action of $\tilde{K}$ on $\tilde{\mathfrak{p}}$ induces a cohomogeneity one action on the projective space $P(\mathfrak{\tilde{p}}) \approx \mathbb{C}P^n.$\\

According to \cite[Theorem 3.1]{PT}, any cohomogeneity one action of a compact connected Lie group on the complex projective space $\mathbb{C}P^n$ is orbit equivalent to an action induced by a Hermitian symmetric pair. This means that the principal orbits of both actions are related by an isometry of $\mathbb{C}P^n$ with respect to the Fubini-Study metric. These actions were classified by Takagi in \cite{T}, where the author established the following result.

\begin{theorem} Let $G$ be a compact connected Lie group acting on a complex projective space with cohomogeneity one. Then the action of $G$ is orbit equivalent to the action induced by a Hermitian symmetric pair $(\tilde{U},\tilde{K}),$ whose corresponding Lie algebras are given in Table \ref{table:1}.

\end{theorem}

In what follows, for each model in Table \ref{table:1}, we fix a representative cohomogeneity one action and explicitly describe the principal and singular orbits and their isotropy representations. For each model, we also introduce a transversal curve, denoted by $\gamma$, which intersects all the orbits of the action and allows us to determine the principal and singular isotropy groups. These curves should not be confused with the normal geodesic $c$ considered in Section~\ref{section:2}. For Models {\bf A}-{\bf C}, our exposition closely follows \cite[\S 3]{U}.

\subsection{Model A}  Fix integers $p\geq q\geq 1$ with $p>1.$ Let  \[G_{p,q}:=\mathrm{U}(p+1)\times \mathrm{U}(q+1)\] 
act on $\mathbb{C}P^{p+q+1}=\{[z:w]:z\in\mathbb{C}^{p+1},\ w\in\mathbb{C}^{q+1},\ |z|^2+|w|^2\neq 0\}$ with cohomogeneity one by the formula $$(U,V)\cdot[z:w]:=[Uz:Vw].$$ On $\mathfrak{g}_{p,q}:=\mathfrak{u}(p+1)\oplus\mathfrak{u}(q+1),$ consider  the $\Ad(G_{p,q})$-invariant inner product $(\cdot,\cdot)$ defined by
\begin{align}\label{inner:product:model:A}
    \begin{split}
    \left((Z_1,Z_2),(Z_3,Z_4)\right):=&\frac{\mathrm{Tr}(Y_1Y_3-X_1X_3)+\mathrm{Tr}(Y_2Y_4-X_2X_4)}{2},
    \end{split}
\end{align}
where $Z_i=X_i+\sqrt{-1}Y_i$ and $X_i,Y_i$ have real entries.\\

Define the curve $\gamma:[0,1]\rightarrow\mathbb{C}P^{p+q+1}$ by $$\gamma(t):=[0:\cdots:0:1-t:0:\cdots:0:t],$$
where $1-t$ and $t$ appear in the $(p+1)$-th and $(p+q+2)$-th positions, respectively. For $0<t<1,$ the isotropy group at $\gamma(t)$ is
\begin{align*}
    H&=\left(\mathrm{U}(p+1)\times\mathrm{U}(q+1)\right)_{\gamma(t)}\\
    &=\left\{(U,V)\in\mathrm{U}(p+1)\times\mathrm{U}(q+1):U=\mathrm{diag}(U_1,\lambda),\ V=\mathrm{diag}(V_1,\lambda),\ \lambda\in\mathrm{U}(1)\right\}\\
    &\cong
\mathrm{U}(p)\times\mathrm{U}(q)\times\mathrm{U}(1).
\end{align*}
The isotropy groups at $t=0$ and $t=1$ are
\begin{align*}
     K_0&=\left(\mathrm{U}(p+1)\times\mathrm{U}(q+1)\right)_{\gamma(0)}\\
     &=\left\{(U,V)\in\mathrm{U}(p+1)\times\mathrm{U}(q+1):U=\mathrm{diag}(U_1,\lambda),\ \lambda\in\mathrm{U}(1)\right\}\\
    &\cong \mathrm{U}(p)\times \mathrm{U}(1)\times\mathrm{U}(q+1),
\end{align*}
and
\begin{align*}
     K_1&=\left(\mathrm{U}(p+1)\times\mathrm{U}(q+1)\right)_{\gamma(1)}\\
     &=\left\{(U,V)\in\mathrm{U}(p+1)\times\mathrm{U}(q+1):V=\mathrm{diag}(V_1,\lambda),\ \lambda\in\mathrm{U}(1)\right\}\\
    &\cong \mathrm{U}(p+1)\times\mathrm{U}(q)\times \mathrm{U}(1).
\end{align*}
Thus, the singular orbits are 
$$Q_0=\frac{\mathrm{U}(p+1)\times\mathrm{U}(q+1)}{\mathrm{U}(p)\times\mathrm{U}(1)\times\mathrm{U}(q+1)}\approx\frac{\mathrm{U}(p+1)}{\mathrm{U}(p)\times\mathrm{U}(1)}\approx \mathbb{C}P^p$$
and
$$Q_1=\frac{\mathrm{U}(p+1)\times\mathrm{U}(q+1)}{\mathrm{U}(p+1)\times\mathrm{U}(q)\times\mathrm{U}(1)}\approx\frac{\mathrm{U}(q+1)}{\mathrm{U}(1)\times\mathrm{U}(q)}\approx \mathbb{C}P^q.$$
The Lie algebra of $H$ is $\mathfrak{h}\cong\mathfrak{u}(p)\oplus\mathfrak{u}(q)\oplus\mathbb{R}.$ Its $(\cdot,\cdot)$-orthogonal complement $\mathfrak{m}$ in $\mathfrak{g}_{p,q}$ splits into a sum of three $\Ad(H)$-invariant, irreducible, pairwise $(\cdot,\cdot)$-orthogonal subspaces, namely,
\begin{align}\label{summands:model:A}
    \begin{split}
    &\mathfrak{m}_1=\Span_{\mathbb{R}}\{(\sqrt{-1}E_{p+1,p+1},-\sqrt{-1}F_{q+1,q+1})\},\\
    &\mathfrak{m}_2=\Span_{\mathbb{R}}\{(E_{p+1,j}-E_{j,p+1},0),(\sqrt{-1}(E_{p+1,j}+E_{j,p+1}),0):1\leq j\leq p\},\\
    &\mathfrak{m}_3=\Span_{\mathbb{R}}\{(0,F_{q+1,j}-F_{j,q+1}),(0,\sqrt{-1}(F_{q+1,j}+F_{j,q+1})):1\leq j\leq q\}, 
    \end{split}
\end{align}
where $E_{ij}$ (respectively, $F_{ij}$) denotes the $(p+1)\times(p+1)$ (respectively, $(q+1)\times(q+1)$) matrix with a 1 in the $(i,j)$-entry and zeros elsewhere. Set $$\mathfrak{p}_0:=\mathfrak{m}_1\oplus\mathfrak{m}_3,\ \mathfrak{n}_0:=\mathfrak{m}_2,\ \mathfrak{p}_1:=\mathfrak{m}_1\oplus\mathfrak{m}_2\ \textnormal{and}\ \mathfrak{n}_1:=\mathfrak{m}_3.$$ If $\mathfrak{k}_0$ and $\mathfrak{k}_1$ are the Lie algebras corresponding to $K_0$ and $K_1$ respectively, then  $$\mathfrak{k}_i=\mathfrak{h}\oplus\mathfrak{p}_i,\ (i=0,1)$$ is a reductive $(\cdot,\cdot)$-orthogonal decomposition for the homogeneous space $K_i/H\approx\mathbb{S}^{\ell_i},$ where $\ell_0=2q+1$ and $\ell_1=2p+1.$ Additionally,  $$\mathfrak{g}_{p,q}=\mathfrak{k}_i\oplus\mathfrak{n}_i,\ (i=0,1)$$
is a reductive, $(\cdot,\cdot)$-orthogonal decomposition of $Q_i.$ Let $$\{e_1,\sqrt{-1}e_1,...,e_{p+1},\sqrt{-1}e_{p+1}\}\ \textnormal{and}\ \{f_1,\sqrt{-1}f_1,...,f_{q+1},\sqrt{-1}f_{q+1}\}$$ be the canonical bases for the real spaces $\mathbb{C}^{p+1}$ and $\mathbb{C}^{q+1},$ respectively. Then, the maps $\psi_i:K_i/H\to\mathbb{S}^{\ell_i}$ given by 
\begin{equation}\label{diffeomorphisms:spheres:model:A}
\psi_0((U,V)H):=\lambda^{-1}Vf_{q+1}\ \textnormal{and}\ \psi_1((\tilde{U},\tilde{V})H):=\tilde{\lambda}^{-1}\tilde{U}e_{p+1},
\end{equation}
where $(U,V)\in K_0$ with $U=\mathrm{diag}(U_1,\lambda)$ and
$(\tilde{U},\tilde{V})\in K_1$ with $\tilde{V}=\mathrm{diag}(\tilde{V}_1,\tilde{\lambda}),$
are diffeomorphisms and their derivatives at the identity class are \begin{align}\label{differential:psi_0:model:A}
\begin{split}
    &(d\psi_0)_H\big(\sqrt{-1}E_{p+1,p+1},-\sqrt{-1}F_{q+1,q+1}\big)=-\,2\sqrt{-1}\,f_{q+1},\\
    &(d\psi_0)_H\big(0,F_{q+1,r}-F_{r,q+1}\big)=f_r,\\
    &(d\psi_0)_H\big(0,\sqrt{-1}(F_{q+1,r}+F_{r,q+1})\big)=\sqrt{-1}\,f_r,\ 1\le r\le q,
\end{split}
\end{align}
and
\begin{align}\label{differential:psi_1:model:A}
\begin{split}
&(d\psi_1)_H\big(\sqrt{-1}E_{p+1,p+1},-\sqrt{-1}F_{q+1,q+1}\big)=-\,2\sqrt{-1}\,e_{p+1},\\
&(d\psi_1)_H\big(E_{p+1,j}-E_{j,p+1},0\big)=e_j,\\
&(d\psi_1)_H\big(\sqrt{-1}(E_{p+1,j}+E_{j,p+1}),0\big)=\sqrt{-1}\,e_j,\ 1\le j\le p.
\end{split}
\end{align}

\subsection{Model B}
Fix $n\ge2$ and let $\mathrm{SU}(2)\times \mathrm{SU}(n+1)$ act on $\mathbb{C}P^{2n+1}=P(\mathbb{C}^2\otimes \mathbb{C}^{n+1})$ via the action induced by the tensor product representation
\[
(A,B)\cdot [v\otimes w]:=[Av\otimes Bw],\ (A,B)\in \mathrm{SU}(2)\times \mathrm{SU}(n+1).
\]
On $\mathfrak{su}(2)\oplus \mathfrak{su}(n+1),$ the Killing form $B$ is given by 
\begin{equation*}
    B((X_1,Y_1),(X_2,Y_2)):=4\mathrm{Tr}(X_1X_2)+2(n+1)\mathrm{Tr}(Y_1Y_2).
\end{equation*}
We fix the $\Ad(\mathrm{SU}(2)\times\mathrm{SU}(n+1))$-invariant inner product
\begin{equation}\label{inner:product:model:B}
(\cdot,\cdot):=-B.
\end{equation}

Let $\{e_1,e_2\}$ be the standard basis of $\mathbb{C}^2$, and let 
$\{f_1,\ldots,f_{n+1}\}$ be the standard basis of $\mathbb{C}^{\,n+1}$.
Consider the curve
\[
\gamma:[0,1]\to \mathbb{C}P^{2n+1},\
\gamma(t):=\big[\,t\,e_1\otimes f_{n}+e_2\otimes f_{n+1}\,\big].
\]
For $0<t<1$, the isotropy subgroup is
\begin{align*}
H=&\left(\mathrm{SU}(2)\times\mathrm{SU}(n+1)\right)_{\gamma(t)}\\
=&\left\{\left(\overline{U},\mathrm{diag}(V,\lambda U)\right):U\in\mathrm{S}(\mathrm{U}(1)\times\mathrm{U}(1)),\ V\in\mathrm{U}(n-1),\ \det(V)\lambda^2=1\right\}.
\end{align*}
At the boundary points, the isotropy groups are
\begin{align*}
K_0&=\left(\mathrm{SU}(2)\times \mathrm{SU}(n+1)\right)_{\gamma(0)}\\[2pt]
&=\left\{\left(U,\mathrm{diag}(V,\lambda)\right):U\in\mathrm{S}\left(\mathrm{U}(1)\times\mathrm{U}(1)\right),\ V\in\mathrm{U}(n),\ \det(V)\lambda=1\right\}\\[2pt]
&\cong\mathrm{S}\left(\mathrm{U}(1)\times\mathrm{U}(1)\right)\times \mathrm{S}\left(\mathrm{U}(n)\times \mathrm{U}(1)\right),\\[2pt]
K_1
&=\left(\mathrm{SU}(2)\times \mathrm{SU}(n+1)\right)_{\gamma(1)}\\[2pt]
&=\left\{\left(
\overline{U},\mathrm{diag}(V,\lambda U)\right)\ :U\in\mathrm{SU}(2),\ V\in \mathrm{U}(n-1),\ \det(V)\lambda^2=1\right\}
\end{align*}
and the corresponding singular orbits are
\[
\displaystyle Q_0=\frac{\mathrm{SU}(2)\times\mathrm{SU}(n+1)}{K_0}\approx \mathbb{C}P^{1}\times\mathbb{C}P^{n},\
Q_1=\frac{\mathrm{SU}(2)\times\mathrm{SU}(n+1)}{K_1}.
\]
Let $E_{ij}$ denote the $2\times 2$ matrix with a $1$ in the $(i,j)$-entry and zeros elsewhere, and let
$F_{ij}$ denote the $(n+1)\times (n+1)$ matrix with a $1$ in the $(i,j)$-entry and zeros elsewhere. Consider the following elements of $\mathfrak{su}(2)\oplus \mathfrak{su}(n+1)$:
\begin{align*}
X_{11}&:=\left(\sqrt{-\frac{n+1}{8(n+3)}}\,(E_{11}-E_{22}),\ \sqrt{\frac{-1}{2(n+1)(n+3)}}\,(F_{nn}-F_{n+1,n+1})\right),\\[0.4em]
X_{21}&:=\left(\frac{E_{21}-E_{12}}{2\sqrt{n+3}},\ \frac{F_{n+1,n}-F_{n,n+1}}{2\sqrt{n+3}}\right),\\[0.4em]
X_{22}&:=\left(\frac{\sqrt{-1}(E_{21}+E_{12})}{2\sqrt{n+3}},\ -\,\frac{\sqrt{-1}(F_{n+1,n}+F_{n,n+1})}{2\sqrt{n+3}}\right),\\[0.4em]
X_{31}&:=\left(\sqrt{\frac{n+1}{8(n+3)}}\,(E_{21}-E_{12}),\ -\,\sqrt{\frac{1}{2(n+1)(n+3)}}\,(F_{n+1,n}-F_{n,n+1})\right),\\[0.4em]
X_{32}&:=\left(\sqrt{-\frac{n+1}{8(n+3)}}\,(E_{21}+E_{12}),\ \sqrt{\frac{-1}{2(n+1)(n+3)}}\,(F_{n+1,n}+F_{n,n+1})\right),
\end{align*}
and, for each $j=1,\dots,n-1$,
\begin{align*}
X_{4j}&:=\left(0,\ \frac{F_{n j}-F_{j n}}{2\sqrt{n+1}}\right),
\
Y_{4j}:=\left(0,\ \frac{\sqrt{-1}\,(F_{n j}+F_{j n})}{2\sqrt{n+1}}\right),\\[0.4em]
X_{5j}&:=\left(0,\ \frac{F_{n+1,j}-F_{j,n+1}}{2\sqrt{n+1}}\right),
\
Y_{5j}:=\left(0,\ \frac{\sqrt{-1}\,(F_{n+1,j}+F_{j,n+1})}{2\sqrt{n+1}}\right).
\end{align*}

With respect to the inner product \eqref{inner:product:model:B}, the Lie algebra
$\mathfrak{su}(2)\oplus \mathfrak{su}(n+1)$ admits a reductive $(\cdot,\cdot)$-orthogonal decomposition
\begin{equation*}
\mathfrak{su}(2)\oplus\mathfrak{su}(n+1)
=\mathfrak{h}\oplus\mathfrak{m}_1\oplus\mathfrak{m}_2\oplus\ \mathfrak{m}_3\oplus\mathfrak{m}_4\oplus\mathfrak{m}_5,
\end{equation*}
where $\mathfrak{h}$ is the Lie algebra of $H$ and
\begin{align}\label{reductive:decomposition:explicit:model:B}
\begin{split}
\mathfrak{m}_1&=\Span_{\mathbb{R}}\{X_{11}\},\\
\mathfrak{m}_2&=\Span_{\mathbb{R}}\{X_{21},X_{22}\},\\
\mathfrak{m}_3&=\Span_{\mathbb{R}}\{X_{31},X_{32}\},\\
\mathfrak{m}_4&=\Span_{\mathbb{R}}\{X_{4j},Y_{4j}:j=1,\dots,n-1\},\\
\mathfrak{m}_5&=\Span_{\mathbb{R}}\{X_{5j},Y_{5j}:j=1,\dots,n-1\}.
\end{split}
\end{align}
Let
\[
\mathfrak{p}_0=\mathfrak{m}_1\oplus\mathfrak{m}_4,\
\mathfrak{n}_0=\mathfrak{m}_2\oplus\mathfrak{m}_3\oplus\mathfrak{m}_5,\
\mathfrak{p}_1=\mathfrak{m}_2,\
\mathfrak{n}_1=\mathfrak{m}_1\oplus\mathfrak{m}_3\oplus\mathfrak{m}_4\oplus\mathfrak{m}_5.
\]
If $\mathfrak{h}$ and $\mathfrak{k}_i$ $(i=0,1)$ denote the Lie algebras of $H$ and $K_i$, respectively, then
\[
\mathfrak{k}_i=\mathfrak{h}\oplus\mathfrak{p}_i,
\ \mathfrak{su}(2)\oplus\mathfrak{su}(n+1)=\mathfrak{k}_i\oplus\mathfrak{n}_i,
\ i=0,1.
\]
Moreover, the decomposition $\mathfrak{k}_i=\mathfrak{h}\oplus \mathfrak{p}_i$ is reductive for the homogeneous space
$K_i/H\approx \mathbb{S}^{\ell_i}$, where $\ell_0=2n-1$ and $\ell_1=2$. In this model, there are diffeomorphisms
$\psi_i:K_i/H\to\mathbb{S}^{\ell_i}$ given by
\begin{align*}
\psi_0\left(\left(\mathrm{diag}(\alpha,\alpha^{-1}),\mathrm{diag}(V,\lambda)\right)H\right)
&:=\frac{\alpha^2}{\lambda}\,V\tilde{e}_n,
\end{align*}
and
\begin{align*}
\psi_1\left(\left(\overline{U},\mathrm{diag}(V,\lambda U)\right)H\right)
&:=\frac{\sqrt{-1}}{2\sqrt{2}}U\left(E_{11}-E_{22}\right)U^{-1},
\end{align*}
where $\{\tilde{e}_1,...,\tilde{e}_n\}$ is the standard basis of $\mathbb{C}^n,$
\begin{align*}
&\mathbb{S}^{2n-1}=\left\{(w_1,\dots,w_n)\in\mathbb{C}^{n}:\ |w_1|^2+\cdots+|w_n|^2=1\right\},\ \textnormal{and}\\
&\mathbb{S}^2=\left\{X\in\mathfrak{su}(2):\ -\mathrm{Tr}(X^2)=\frac{1}{4}\right\}.
\end{align*}
The derivatives of $\psi_0,\ \psi_1$ at the identity class are given by
\begin{align}\label{derivatives:psi0:model:B}
\begin{split}
(d\psi_0)_H(X_{11})
&=\sqrt{\frac{n+3}{2(n+1)}}\,\sqrt{-1}\,\tilde{e}_n,\\[0.4em]
(d\psi_0)_H(X_{4j})
&=-\frac{1}{2\sqrt{n+1}}\,\tilde{e}_j,\\[0.4em]
(d\psi_0)_H(Y_{4j})
&=\frac{\sqrt{-1}}{2\sqrt{n+1}}\,\tilde{e}_j,
\ 1\le j\le n-1,
\end{split}
\end{align}
and
\begin{align}\label{derivatives:psi1:model:B}
\begin{split}
&(d\psi_1)_H(X_{21})=\frac{\sqrt{-1}}{2\sqrt{2(n+3)}}(E_{21}+E_{12}),\\[0.5em]
&(d\psi_1)_H(X_{22})
=-\frac{1}{2\sqrt{2(n+3)}}(E_{21}-E_{12}).
\end{split}
\end{align}

\subsection{Model C} For any natural number $n\geq 2,$ the special orthogonal group $\SO(n+1)$ acts with cohomogeneity one on the complex projective space $\mathbb{C}P^n.$ The action is defined as $U\cdot[z]=\left[Uz\right],$ where $U\in\SO(n+1),$ $z\in\mathbb{C}^{n+1}-\{0\}$ and $[z]$ denotes the class of $z$ in $\mathbb{C}P^{n}.$ Let $\gamma:[0,1]\rightarrow \mathbb{C}P^{n}$ be the curve defined by  $$\gamma(t):=\left[0:\cdots:0:1-t:\sqrt{-1}\right]$$ and fix the $\Ad(\SO(n+1))$-invariant inner product $(\cdot,\cdot)$ on $\mathfrak{so}(n+1)$ given by \begin{equation}\label{inner:product:model:C}
    (X,Y):=-2\mathrm{Tr}(XY),\ X,Y\in\mathfrak{so}(n+1).
\end{equation}
The isotropy group of $\gamma(t)$ is 
\begin{align*}
&\begin{aligned}
    H:=\SO(n+1)_{\gamma(t)}&=\left\{\left(\begin{array}{cc}
       V  & 0 \\
       0  & \epsilon\,\mathrm{Id}_2
    \end{array}\right)\in\SO(n+1):V\in\SO(n-1)\ \textnormal{and}\ \epsilon=\pm 1\right\}\\[0.2em]
    &\cong \SO(n-1)\times \mathbb{Z}_2,\ \textnormal{if}\ 0<t<1,
\end{aligned}\\
&\begin{aligned}
     K_0:=\SO(n+1)_{\gamma(0)}&=\left\{\left(\begin{array}{cc}
       V_1  & 0 \\
       0  & V_2
    \end{array}\right)\in\SO(n+1):V_1\in\SO(n-1)\ \textnormal{and}\ V_2\in\SO(2)\right\}\\[0.2em]
    &\cong \SO(n-1)\times\SO(2),
\end{aligned}\\
&\begin{aligned}
     K_1:=\SO(n+1)_{\gamma(1)}&=\left\{\left(\begin{array}{cc}
       V  & 0 \\
       0  & \epsilon
    \end{array}\right)\in\SO(n+1):V\in\mathrm{O}(n)\ \textnormal{and}\ \epsilon=\det(V)\right\}\\[0.2em]
    &\cong \mathrm{S}(\mathrm{O}(n)\times \mathrm{O}(1)).
\end{aligned}
\end{align*}

Their corresponding Lie algebras are 
\begin{align*}
    &\mathfrak{h}=\Span_{\mathbb{R}}\{E_{ij}-E_{ji}:1\leq j<i\leq n-1\}\cong \mathfrak{so}(n-1),\\
    &\mathfrak{k}_0=\Span_{\mathbb{R}}\{E_{ij}-E_{ji}:1\leq j<i\leq n-1\}\cup\{E_{n+1,n}-E_{n,n+1}\}\cong \mathfrak{so}(n-1)\oplus\mathfrak{so}(2),\\
    &\mathfrak{k}_1=\Span_{\mathbb{R}}\{E_{ij}-E_{ji}:1\leq j<i\leq n\}\cong \mathfrak{so}(n),
\end{align*}
where $E_{ij}$ denotes the $(n+1)\times(n+1)$ matrix with 1 in the $(i,j)$-entry and zeros elsewhere. A principal orbit $P$ is diffeomorphic to the homogeneous space $\frac{\SO(n+1)}{\SO(n-1)\times \mathbb{Z}_2}$ and its isotropy representation decomposes into the three irreducible, pairwise $(\cdot,\cdot)$-orthogonal summands 
\begin{align}\label{summands:model:C}
    \begin{split}
    &\mathfrak{m}_1=\Span_{\mathbb{R}}\left\{E_{n+1,n}-E_{n,n+1}\right\},\\
    &\mathfrak{m}_2=\Span_{\mathbb{R}}\left\{E_{nj}-E_{jn}:1\leq j\leq n-1\right\},\\
    &\mathfrak{m}_3=\Span_{\mathbb{R}}\left\{E_{n+1,j}-E_{j,n+1}:1\leq j\leq n-1\right\}.
    \end{split}
\end{align}
The singular orbits are 
\begin{equation}\label{singular:orbits:model:C}
    Q_1=\frac{\SO(n+1)}{\mathrm{S}\left(\mathrm{O}(n)\times\mathrm{O}(1)\right)}\approx \mathbb{R}P^n\ \textnormal{and}\ Q_0=\frac{\SO(n+1)}{\SO(n-1)\times\SO(2)}
\end{equation}
and their corresponding $(\cdot,\cdot)$-orthogonal reductive decompositions are
$$\mathfrak{so}(n+1)=\mathfrak{k}_{i}\oplus\mathfrak{n}_{i},\ i=0,1,$$
where $\mathfrak{n}_1=\mathfrak{m}_1\oplus\mathfrak{m}_3$ and $\mathfrak{n}_0=\mathfrak{m}_2\oplus\mathfrak{m}_3.$ If we denote $\mathfrak{p}_1=\mathfrak{m}_2$ and $\mathfrak{p}_0=\mathfrak{m}_1$ then 
\begin{align*}    \mathfrak{k}_{i}=\mathfrak{h}\oplus\mathfrak{p}_{i}=\mathfrak{so}(n-1)\oplus\mathfrak{p}_{i},\ i=0,1,
\end{align*} is also a reductive $(\cdot,\cdot)$-orthogonal decomposition. The quotient $K_{i}/H,\ i=0,1$ is diffeomorphic to the sphere $\mathbb{S}^{\ell_{i}},$ where $\ell_0=1$ and $\ell_1=n-1.$ The map $$\SO(2)\ni\left(\begin{array}{cc}\cos \theta & -\sin \theta\\ \sin \theta & \cos \theta\end{array}\right)\longmapsto (\cos(2\theta),\sin(2\theta))\in \mathbb{S}^1$$
induces a diffeomorphism $\psi_0$ from $K_0/H=\frac{\SO(n-1)\times \SO(2)}{\SO(n-1)\times\mathbb{Z}_2}\approx \frac{\SO(2)}{\mathbb{Z}_2}$ to $\mathbb{S}^1$ such that \begin{equation}\label{derivative:psi0:model:C}
    \left(d\psi_0\right)_H\left(E_{n+1,n}-E_{n,n+1}\right)=(0,2).
\end{equation}

On the other hand, a diffeomorphism between $
K_1/H=\frac{\mathrm{S}\left(\mathrm{O}(n)\times \mathrm{O}(1)\right)}{\SO(n-1)\times \mathbb{Z}_2}$ and $\mathbb{S}^{n-1}$ is given by 
\begin{equation*}
    \psi_1:\frac{\mathrm{S}\left(\mathrm{O}(n)\times \mathrm{O}(1)\right)}{\textnormal{SO}(n-1)\times \mathbb{Z}_2}\ni\mathrm{diag}(V,\epsilon)\longmapsto\epsilon Ve_n\in\mathbb{S}^{n-1},
\end{equation*}
where $\{e_1,\dots,e_n\}$ is the standard basis of $\mathbb{R}^n.$ Its derivative at the identity class satisfies
\begin{equation}\label{derivative:psi1:model:C}
    \left(d\psi_1\right)_H\left(E_{nj}-E_{jn}\right)=-e_j\in T_{e_n}\mathbb{S}^{n-1},\ j=1,...,n-1.   
\end{equation}

\subsection{Model D}
Consider the action of $\mathrm{U}(5)$ on $\mathbb{C}^{10}\cong \Lambda^2\mathbb{C}^5$ given by
\[
U\cdot(v\wedge w):=(Uv)\wedge(Uw),\ U\in\mathrm{U}(5).
\]
It induces a cohomogeneity one action on the projective space $\mathbb{C}P^9=P(\mathbb{C}^{10}),$ namely
\[
U\cdot[v\wedge w]=[(Uv)\wedge(Uw)].
\]
Let $\{e_1,\dots,e_5\}$ be the standard basis of $\mathbb{C}^5$, and consider the curve $\gamma:[0,1]\to\mathbb{C}P^9$ defined by
\[
\gamma(t):=[e_1\wedge e_2+t\,e_3\wedge e_4].
\]

Fix the $\Ad(\mathrm{U}(5))$-invariant inner product $(\cdot,\cdot)$ on $\mathfrak{u}(5)$ given by
\begin{equation}\label{inner:product:model:D}
    (X,Y):=-\mathrm{Tr}(XY),\ X,Y\in\mathfrak{u}(5).
\end{equation}
For $0<t<1$, the isotropy group of $\gamma(t)$ is
\[
H:=\mathrm{U}(5)_{\gamma(t)}=\left\{\mathrm{diag}(V_1,V_2,\lambda):V_1,V_2\in\mathrm{U}(2),\ \lambda\in\mathrm{U}(1),\ \det(V_1)=\det(V_2)\right\}.
\]
The singular isotropy groups are
\begin{align*}
    K_0&:=\mathrm{U}(5)_{\gamma(0)}=\left\{\mathrm{diag}(V_1,V_2):V_1\in\mathrm{U}(2),\ V_2\in\mathrm{U}(3)\right\}\cong \mathrm{U}(2)\times\mathrm{U}(3),\\[0.5em]
    K_1&:=\mathrm{U}(5)_{\gamma(1)}=\left\{\mathrm{diag}(V,\lambda):\lambda\in\mathrm{U}(1),\ V\in\mathrm{U}(4),\ VJ_4V^T\in\mathrm{U}(1)\cdot J_4\right\}\\[0.2em]
    &\,\,\cong\left(\mathrm{U}(1)\cdot\mathrm{Sp}(2)\right)\times\mathrm{U}(1),
\end{align*}
where
\[
J_4=\left(\begin{array}{cccc}0&1&0&0\\-1&0&0&0\\0&0&0&1\\0&0&-1&0\end{array}\right).
\]
The singular orbits are
\[
Q_0=\frac{\mathrm{U}(5)}{\mathrm{U}(2)\times\mathrm{U}(3)}\approx\mathrm{Gr}_2(\mathbb{C}^5),
\ Q_1=\frac{\mathrm{U}(5)}{\left(\mathrm{U}(1)\cdot\mathrm{Sp}(2)\right)\times\mathrm{U}(1)}.
\]

The Lie algebras of the isotropy groups are
\begin{align*}
    \mathfrak{h}&=\left\{\mathrm{diag}(X,Y,\sqrt{-1}s):X,Y\in\mathfrak{u}(2),\ s\in\mathbb{R},\ \mathrm{Tr}(X)=\mathrm{Tr}(Y)\right\},\\[0.5em]
    \mathfrak{k}_0&=\left\{\mathrm{diag}(X,Y):X\in\mathfrak{u}(2),\ Y\in\mathfrak{u}(3)\right\}\cong\mathfrak{u}(2)\oplus\mathfrak{u}(3),\\[0.5em]
    \mathfrak{k}_1&=\left\{\mathrm{diag}(X,\sqrt{-1}s):X\in\mathfrak{u}(4),\ s\in\mathbb{R},\ XJ_4+J_4X^T\in\sqrt{-1}\mathbb{R}J_4\right\}\cong \mathfrak{u}(1)\oplus\mathfrak{sp}(2)\oplus\mathfrak{u}(1).
\end{align*}

Let
\[
J_2=\left(\begin{array}{cc}
     0 & 1 \\
     -1 & 0
\end{array}\right),
\]
and let $E_{ij}$ denote the complex $5\times 5$ matrix with $1$ in the $(i,j)$-entry and zeros elsewhere. For a square complex matrix $A$, denote by $A^*$ its Hermitian adjoint, i.e. $A^*=\overline{A}^{\,T}$. A reductive $(\cdot,\cdot)$-orthogonal decomposition of the principal orbit $P=\mathrm{U}(5)/H$ is
\[
\mathfrak{u}(5)=\mathfrak{h}\oplus\left(\mathfrak{m}_1\oplus\mathfrak{m}_2\oplus\mathfrak{m}_3\oplus\mathfrak{m}_4\oplus\mathfrak{m}_5\right),
\]
where
\begin{align}\label{isotropy:summands:model:D}
    \begin{split}
       &\mathfrak{m}_1=\Span_\mathbb{R}\left\{\frac{\sqrt{-1}}{2}(E_{11}+E_{22}-E_{33}-E_{44})\right\},\\[0.2em]
       &\mathfrak{m}_2=\Span_\mathbb{R}\left\{\left(\begin{array}{ccc}0&Z&0\\-Z^*&0&0\\0&0&0\end{array}\right):ZJ_2=J_2\overline{Z}\right\},\\[0.2em]
       &\mathfrak{m}_3=\Span_\mathbb{R}\left\{\left(\begin{array}{ccc}0&Z&0\\-Z^*&0&0\\0&0&0\end{array}\right):ZJ_2=-J_2\overline{Z}\right\},\\[0.2em]
       &\mathfrak{m}_4=\Span_\mathbb{R}\left\{E_{5i}-E_{i5},\sqrt{-1}(E_{5i}+E_{i5}):i\in\{1,2\}\right\},\\[0.2em]
       &\mathfrak{m}_5=\Span_\mathbb{R}\left\{E_{5i}-E_{i5},\sqrt{-1}(E_{5i}+E_{i5}):i\in\{3,4\}\right\}.
    \end{split}
\end{align}
Setting
\[
\mathfrak{p}_0=\mathfrak{m}_1\oplus\mathfrak{m}_5,\ \mathfrak{n}_0=\mathfrak{m}_2\oplus\mathfrak{m}_3\oplus\mathfrak{m}_4,\ \mathfrak{p}_1=\mathfrak{m}_2,\ \mathfrak{n}_1=\mathfrak{m}_1\oplus\mathfrak{m}_3\oplus\mathfrak{m}_4\oplus\mathfrak{m}_5,
\]
we obtain the reductive decompositions
\[
\mathfrak{u}(5)=\mathfrak{k}_i\oplus\mathfrak{n}_i,\ \mathfrak{k}_i=\mathfrak{h}\oplus\mathfrak{p}_i,\ i=0,1.
\]
In particular, $K_i/H\approx \mathbb{S}^{\ell_i}$ for $i=0,1$, where $\ell_0=5$ and $\ell_1=4$.\\

To define a diffeomorphism $\psi_0:K_0/H\to\mathbb{S}^5$, let $\{f_1,f_2,f_3\}$ be the standard basis of $\mathbb{C}^3$ and consider
\[
\Lambda^2\mathbb{C}^3=\Span_\mathbb{C}\left\{f_1\wedge f_2,\ f_1\wedge f_3,\ f_2\wedge f_3\right\},
\]
which contains the five-dimensional sphere
\[
\mathbb{S}^5=\left\{z_{12}(f_1\wedge f_2)+z_{13}(f_1\wedge f_3)+z_{23}(f_2\wedge f_3):\ |z_{12}|^2+|z_{13}|^2+|z_{23}|^2=1\right\}.
\]
Define $\psi_0:K_0/H\to\mathbb{S}^5$ by
\[
\psi_0(\mathrm{diag}(V_1,V_2)H):=\det(V_1)^{-1}(V_2f_1)\wedge(V_2f_2).
\]
Its derivative at the identity class is given by
\begin{align}\label{derivatives:psi_0:model:D}
    \begin{split}
    &(d\psi_0)_H\left(\frac{\sqrt{-1}}{2}(E_{11}+E_{22}-E_{33}-E_{44})\right)=-2\sqrt{-1}\,f_1\wedge f_2,\\[0.2em]
    &(d\psi_0)_H\left(E_{54}-E_{45}\right)=f_1\wedge f_3,\ (d\psi_0)_H(\sqrt{-1}(E_{54}+E_{45}))=\sqrt{-1}\,f_1\wedge f_3,\\[0.2em]
    &(d\psi_0)_H(E_{53}-E_{35})=-f_2\wedge f_3,\ (d\psi_0)_H(\sqrt{-1}(E_{53}+E_{35}))=-\sqrt{-1}\,f_2\wedge f_3.
    \end{split}
\end{align}

For $K_1/H\approx\mathbb{S}^{4}$, let $\{\tilde{f}_1,\tilde{f}_2\}$ be the standard basis of $\mathbb{C}^2$. Given $\mathrm{diag}(V,\lambda)\in K_1$, we have
\[
VJ_4V^T=wJ_4
\]
for some $w\in\mathrm{U}(1)$. Choose $z\in\mathrm{U}(1)$ with $z^2=w$ and set $S:=z^{-1}V$, so that
\[
S\in\mathrm{Sp}(2)=\{A\in\mathrm{U}(4):AJ_4A^T=J_4\}.
\]
Write
\[
S=\left(\begin{array}{cc}A &B\\ C& D\end{array}\right),\ A,B,C,D\in M_{2\times2}(\mathbb{C}),
\]
and set
\[
u:=A\tilde{f}_1=\left(\begin{array}{c}u_1\\u_2\end{array}\right),\
v:=C\tilde{f}_1=\left(\begin{array}{c}v_1\\v_2\end{array}\right)\in\mathbb{C}^2
\]
the components of the first column of $S$. Then a diffeomorphism $\psi_1:K_1/H\to\mathbb{S}^4$ is given by
\[
\psi_1(\mathrm{diag}(V,\lambda)H):=(x_0,x_1,x_2,x_3,x_4),
\]
where
\begin{align*}
&x_0=|u_1|^2+|u_2|^2-|v_1|^2-|v_2|^2,\\
&x_1=-2\,\mathrm{Re}(u_1\overline{v_1}+u_2\overline{v_2}),\\
&x_2=-2\,\mathrm{Im}(u_1\overline{v_1}+u_2\overline{v_2}),\\
&x_3=-2\,\mathrm{Re}(u_1v_2-u_2v_1),\\
&x_4=-2\,\mathrm{Im}(u_1v_2-u_2v_1).
\end{align*}
Moreover, the derivative of $\psi_1$ at the identity class is
\begin{align}\label{derivative:psi_1:model:D}
    \begin{split}
        &(d\psi_1)_H\left(E_{13}+E_{24}-E_{31}-E_{42}\right)=2(0,1,0,0,0),\\[0.2em]
        &(d\psi_1)_H\left(\sqrt{-1}(E_{13}-E_{24}+E_{31}-E_{42})\right)=2(0,0,1,0,0),\\[0.2em]
        &(d\psi_1)_H\left(E_{14}-E_{23}+E_{32}-E_{41}\right)=2(0,0,0,1,0),\\[0.2em]
        &(d\psi_1)_H\left(\sqrt{-1}(E_{14}+E_{23}+E_{32}+E_{41})\right)=-2(0,0,0,0,1).
    \end{split}
\end{align}
\subsection{Model E}
Consider the ten-dimensional Euclidean space $\mathbb{R}^{10}$ endowed with the standard inner product $q$ and let
$\{e_1,\dots,e_{10}\}$ be its standard basis. Extending scalars, we regard $q$ as a complex-bilinear form on the complexification
\[
(\mathbb{R}^{10})_{\mathbb{C}}=\mathbb{R}^{10}\otimes_{\mathbb{R}}\mathbb{C}\cong\mathbb{C}^{10}
=\Span_{\mathbb{C}}\{e_1,\dots,e_{10}\}.
\]
Define
\begin{align*}
&u_{k}:=\frac{1}{\sqrt{2}}\left(e_{2k-1}-\sqrt{-1}\,e_{2k}\right),\ v_k:=\frac{1}{\sqrt{2}}\left(e_{2k-1}+\sqrt{-1}\,e_{2k}\right),\ k=1,2,3,4,5.\\[0.2em]
&W:=\Span_{\mathbb{C}}\{u_1,u_2,u_3,u_4,u_5\},\ V:=\Span_{\mathbb{C}}\{v_1,v_2,v_3,v_4,v_5\},
\end{align*}
so that $\mathbb{C}^{10}=W\oplus V$. Equivalently,
\[
W=\ker(J_0-\sqrt{-1}\,\mathrm{Id}_{\mathbb{C}^{10}}),\
V=\ker(J_0+\sqrt{-1}\,\mathrm{Id}_{\mathbb{C}^{10}}),
\]
where $J_0e_{2k-1}=e_{2k}$ and $J_0e_{2k}=-e_{2k-1}$ for $k=1,2,3,4,5$. Denote by $Cl(\mathbb{C}^{10})$ the complex Clifford algebra associated with $(\mathbb{C}^{10},q)$. The algebra
$Cl(\mathbb{C}^{10})$ acts on the exterior algebra
\[
\Lambda^*W=\bigoplus_{r=0}^5\Lambda^rW
\]
as follows. For $k\in\{1,2,3,4,5\}$, define
\begin{align*}
&\varepsilon_k:\Lambda^*W\to\Lambda^*W,\ \varepsilon_k(\omega):=u_k\wedge \omega,\\
&\iota_k:\Lambda^*W\to\Lambda^*W,\ \iota_k(u_{i_1}\wedge\cdots\wedge u_{i_r})
:=\sum_{m=1}^r(-1)^{m-1}\delta_k^{i_m}\,u_{i_1}\wedge\cdots\wedge \widehat{u_{i_m}}\wedge\cdots\wedge u_{i_r},\\
&\rho(e_{2k-1}):=\varepsilon_k+\iota_k,\ \rho(e_{2k}):=\sqrt{-1}(\varepsilon_k-\iota_k),
\end{align*}
where $\widehat{u_{i_m}}$ indicates that the factor $u_{i_m}$ is omitted. For a generic element
\[
g=\sum_{I}a_I\,e_{i_1}\cdots e_{i_s}\in Cl(\mathbb{C}^{10}),
\]
with $I=(i_1,...,i_s),\ 1\leq i_1<\cdots<i_s\leq 10,\ 0\leq s\leq 10,$ define
\[
\rho(g):=\sum_{I}a_I\,\rho(e_{i_1})\circ\cdots\circ\rho(e_{i_s}).
\]
This yields a representation \(\rho:Cl(\mathbb{C}^{10})\to\mathrm{End}(\Lambda^*W).\) The group $\mathrm{Spin}(10)$ is given by
\[
\mathrm{Spin}(10)=\{\,f_1\cdots f_{2k}: f_i\in\mathbb{R}^{10},\ q(f_i,f_i)=1,\ i=1,\dots,2k,\ k\ge 0\,\},
\]
which is a subgroup of $Cl(\mathbb{R}^{10})^\times\subseteq Cl(\mathbb{C}^{10})^\times.$ For each $g\in\mathrm{Spin}(10)$, the representation preserves the even and odd parts of $\Lambda^*W$, that is,
\begin{align*}
&\rho(g)\left(\Lambda^0W\oplus\Lambda^2W\oplus\Lambda^4W\right)\subseteq \Lambda^0W\oplus\Lambda^2W\oplus\Lambda^4W,\\
&\rho(g)\left(\Lambda^1W\oplus\Lambda^3W\oplus\Lambda^5W\right)\subseteq \Lambda^1W\oplus\Lambda^3W\oplus\Lambda^5W.
\end{align*}
Identifying $\mathbb{C}^{16}\cong\Lambda^0W\oplus\Lambda^2W\oplus\Lambda^4W$, we obtain a representation
\begin{equation}\label{rho:model:E}
\rho:\mathrm{Spin}(10)\to\mathrm{GL}(\mathbb{C}^{16}).
\end{equation}
Moreover, if we equip $\mathbb{C}^{16}=\Lambda^0W\oplus\Lambda^2W\oplus\Lambda^4W$ with the Hermitian product induced by the Hermitian product $\langle\cdot,\cdot\rangle_W$ on $W$, for which $\langle u_k,u_\ell\rangle_W=\delta_k^\ell,$ then $\rho(g)\in\mathrm{SU}(16),$ for all $g\in\mathrm{Spin}(10).$ This fact allows us to define a cohomogeneity one action of $\mathrm{Spin}(10)$ on
$\mathbb{C}P^{15}=P(\mathbb{C}^{16})=P(\Lambda^0W\oplus\Lambda^2W\oplus\Lambda^4W)$ by
\[
g\cdot[\omega]:=[\rho(g)\omega].
\]
Let $\gamma:[0,1]\to\mathbb{C}P^{15}$ be the curve
$$\gamma(t)=[1+t\,u_1\wedge u_2\wedge u_3\wedge u_4]$$ 
and let $(\cdot,\cdot)$ be the $\Ad(\mathrm{Spin}(10))$-invariant inner product on $\mathfrak{spin}(10)\cong\mathfrak{so}(10)$ given by 
\begin{equation}\label{inner:product:model:E}
    (X,Y):=-\frac{1}{8}\mathrm{Tr}(XY),\ X,Y\in\mathfrak{so}(10).
\end{equation}
For $0<t<1$, the isotropy group of $\gamma(t)$ is
\[
H=(\mathrm{Spin}(10))_{\gamma(t)}\cong \mathrm{SU}(4)\cdot\mathrm{U}(1)
\]
and the singular isotropy groups are
\[
K_0=(\mathrm{Spin}(10))_{\gamma(0)}\cong \mathrm{SU}(5)\cdot \mathrm{U}(1),
\ K_1=(\mathrm{Spin}(10))_{\gamma(1)}\cong \mathrm{Spin}(7)\cdot\mathrm{U}(1),
\]
with corresponding Lie algebras
\[
\mathfrak{h}\cong \mathfrak{su}(4)\oplus\mathfrak{u}(1)\subseteq\mathfrak{so}(10),\ \mathfrak{k}_0\cong\mathfrak{su}(5)\oplus\mathfrak{u}(1)\subseteq\mathfrak{so}(10),\ \mathfrak{k}_1\cong\mathfrak{spin}(7)\oplus\mathfrak{u}(1)\subseteq\mathfrak{so}(10).
\]

Next we describe explicitly the Lie algebras $\mathfrak{h}$, $\mathfrak{k}_0$, and $\mathfrak{k}_1$ as subspaces of $\mathfrak{so}(10)$.
Let $E_{ij}$ be the $10\times 10$ matrix with $1$ in the $(i,j)$-entry and zeros elsewhere, and set $X_{ij}:=E_{ij}-E_{ji}$. Then
\begin{align*}
&\begin{aligned}
\mathfrak{h}
=\Span_{\mathbb{R}}\{&
X_{2i-1,2j-1}+X_{2i,2j},\ X_{2i-1,2j}-X_{2i,2j-1},\ X_{2k-1,2k}-X_{2k+1,2k+2},\ X_{10,9}:\\
&1\leq j<i\leq 4,\ k=1,2,3\},
\end{aligned}\\
&\begin{aligned}
\mathfrak{k}_0
=\Span_{\mathbb{R}}\{&
X_{2i-1,2j-1}+X_{2i,2j},\ X_{2i-1,2j}-X_{2i,2j-1},\ X_{2k,2k-1}:
1\leq j<i\leq 5,\ k=1,\dots,5\},
\end{aligned}\\
&\begin{aligned}
\mathfrak{k}_1
=\mathfrak{h}\oplus \Span_{\mathbb{R}}\{&
X_{53}-X_{71},\ X_{51}+X_{73},\ X_{75}-X_{31},\ X_{63}+X_{81},\ X_{61}-X_{83},\ X_{41}+X_{85}\}.
\end{aligned}
\end{align*}

For the principal orbit $P=\mathrm{Spin}(10)/(\mathrm{SU}(4)\cdot \mathrm{U}(1))$, we have the reductive $(\cdot,\cdot)$-orthogonal decomposition
\[
\mathfrak{so}(10)=\mathfrak{h}\oplus\mathfrak{m}_1\oplus\mathfrak{m}_2\oplus\mathfrak{m}_3\oplus\mathfrak{m}_4\oplus\mathfrak{m}_5,
\]
where
\begin{align}\label{summands:model:E}
\begin{split}\mathfrak{m}_1&=\Span_\mathbb{R}\left\{\sum_{k=1}^{4}X_{2k,2k-1}\right\},\\[0.2em]
\mathfrak{m}_2&=\Span_\mathbb{R}\left\{X_{10,2k}+X_{9,2k-1},\ X_{10,2k-1}-X_{9,2k}:k=1,2,3,4\right\},\\
\mathfrak{m}_3&=\Span_\mathbb{R}\left\{X_{9,2k-1}-X_{10,2k},\ X_{10,2k-1}+X_{9,2k}:k=1,2,3,4\right\},\\
\mathfrak{m}_4&=\Span_\mathbb{R}\{Y_1,Y_2,Y_3,Y_4,Y_5,Y_6\},\\
\mathfrak{m}_5&=\Span_\mathbb{R}\{Z_1,Z_2,Z_3,Z_4,Z_5,Z_6\},
\end{split}
\end{align}
and
\[
\begin{aligned}
Y_1&:=X_{53}-X_{64}-X_{71}+X_{82},\\
Y_2&:=X_{51}-X_{62}+X_{73}-X_{84},\\
Y_3&:=X_{42}-X_{31}+X_{75}-X_{86},\\
Y_4&:=X_{54}+X_{63}+X_{72}+X_{81},\\
Y_5&:=X_{52}+X_{61}-X_{74}-X_{83},\\
Y_6&:=X_{32}+X_{41}+X_{76}+X_{85},
\end{aligned}
\qquad
\begin{aligned}
Z_1&:=X_{54}+X_{63}-X_{72}-X_{81},\\
Z_2&:=X_{52}+X_{61}+X_{74}+X_{83},\\
Z_3&:=X_{76}+X_{85}-X_{32}-X_{41},\\
Z_4&:=X_{64}-X_{53}-X_{71}+X_{82},\\
Z_5&:=X_{62}-X_{51}+X_{73}-X_{84},\\
Z_6&:=X_{42}-X_{31}-X_{75}+X_{86}.
\end{aligned}
\]

The singular orbits are the homogeneous spaces
\[
Q_0=\frac{\mathrm{Spin}(10)}{\mathrm{SU}(5)\cdot\mathrm{U}(1)},
\ Q_1=\frac{\mathrm{Spin}(10)}{\mathrm{Spin}(7)\cdot \mathrm{U}(1)}.
\]
Set
\[
\mathfrak{p}_0:=\mathfrak{m}_1\oplus\mathfrak{m}_2,\
\mathfrak{n}_0:=\mathfrak{m}_3\oplus\mathfrak{m}_4\oplus\mathfrak{m}_5,\
\mathfrak{p}_1:=\mathfrak{m}_4,\
\mathfrak{n}_1:=\mathfrak{m}_1\oplus\mathfrak{m}_2\oplus\mathfrak{m}_3\oplus\mathfrak{m}_5.
\]
Then the decompositions
\[
\mathfrak{so}(10)=\mathfrak{k}_i\oplus\mathfrak{n}_i,\ \mathfrak{k}_i=\mathfrak{h}\oplus\mathfrak{p}_i,\ i=0,1,
\]
are reductive and $(\cdot,\cdot)$-orthogonal.\\

We now fix an explicit identification of $K_0/H$ with the sphere
\[
\mathbb{S}^9=\{\omega\in\Lambda^4W:||\omega||=1\},
\]
where the norm is taken with respect to the Hermitian inner product on $\Lambda^4W$ induced by $\langle\cdot,\cdot\rangle_W$.
Restricting the representation $\rho$ in \eqref{rho:model:E} to $K_0\subset \mathrm{Spin}(10)$, we obtain the induced unitary action of
$K_0$ on $\Lambda^4W$, and we define
\[
\psi_0:K_0/H\to \mathbb{S}^9\subseteq\Lambda^4W,\ \psi_0(gH):=\rho(g)\,u_1\wedge u_2\wedge u_3\wedge u_4\ (g\in K_0).
\]
This map is a well-defined diffeomorphism, and its differential at the identity class satisfies
\begin{align}\label{derivative:psi0:model:E}
    \begin{split}
        &(d\psi_0)_H\left(\sum_{k=1}^4X_{2k,2k-1}\right)=4\sqrt{-1}\,u_1\wedge u_2\wedge u_3\wedge u_4,\\[0.2em]
        &(d\psi_0)_H\left(X_{10,2k}+X_{9,2k-1}\right)=(-1)^{4-k}u_1\wedge\cdots\wedge \widehat{u_k}\wedge\cdots\wedge u_5,\\[0.2em]
        &(d\psi_0)_H\left(X_{10,2k-1}-X_{9,2k}\right)=\sqrt{-1}(-1)^{4-k}u_1\wedge\cdots\wedge \widehat{u_k}\wedge\cdots\wedge u_5,\ k=1,2,3,4.
    \end{split}
\end{align}
For $K_1/H\approx \left(\mathrm{Spin}(7)\cdot\mathrm{U}(1)\right)/(\mathrm{SU}(4)\cdot\mathrm{U}(1))\approx \mathrm{Spin}(7)/\mathrm{SU}(4)\approx\mathbb{S}^6$, let $\pi:\mathrm{Spin}(7)\to\mathrm{SO}(7)$ be the double covering and denote by $\{f_1,\dots,f_7\}$ the standard basis of $\mathbb{R}^7$. Then the map \[ \mathrm{Spin}(7)/\mathrm{SU}(4)\ni g\,\mathrm{SU}(4)\to \pi(g)f_7\in\mathbb{S}^6, \] 
induces a diffeomorphism 
\[
\psi_1:K_1/H\to\mathbb{S}^6
\]
such that \begin{equation}\label{derivative:psi1:model:E} (d\psi_1)_H(Y_j)=\pm f_j,\ j=1,\dots,6.
\end{equation}
\section{The Einstein Equation on Complex Projective Spaces}\label{Einstein:metrics:CP^n}
\subsection{Model A}
Let $\mathfrak{m}_1,\mathfrak{m}_2,\mathfrak{m}_3$ be as in \eqref{summands:model:A} and $(\cdot,\cdot)$ as in \eqref{inner:product:model:A}. Consider
\begin{align*}
    &\mathcal{B}_1:=\left\{\big(\sqrt{-1}E_{p+1,p+1},-\sqrt{-1}F_{q+1,q+1}\big)\right\},\\
    &\mathcal{B}_2:=\{(E_{p+1,j}-E_{j,p+1},0),\ (\sqrt{-1}(E_{p+1,j}+E_{j,p+1}),0):\ 1\leq j\leq p\},\\
    &\mathcal{B}_3:=\{(0,F_{q+1,j}-F_{j,q+1}),\ (0,\sqrt{-1}(F_{q+1,j}+F_{j,q+1})):\ 1\leq j\leq q\}.
\end{align*}
Then $\mathcal{B}:=\mathcal{B}_1\cup\mathcal{B}_2\cup\mathcal{B}_3$ is an $(\cdot,\cdot)$-orthonormal basis adapted to $\mathfrak{m}_1\oplus\mathfrak{m}_2\oplus\mathfrak{m}_3$. With respect to $\mathcal{B}$, the symbol $[ijk]$ is given by
\[
[ijk]=
\begin{cases}
2p, & \text{if } (i,j,k) \text{ is a permutation of } (1,2,2),\\[4pt]
2q, & \text{if } (i,j,k) \text{ is a permutation of } (1,3,3),\\[4pt]
0,  & \text{otherwise.}
\end{cases}
\]
We now compute the constants $b_i$ in \eqref{equation:b_i} for Model {\bf A}. Recall that the Killing form of $\mathfrak{u}(n)$ is
\[
B_{\mathfrak{u}(n)}(A,B)=2n\,\mathrm{Tr}(AB)-2\,\mathrm{Tr}(A)\mathrm{Tr}(B),\ A,B\in\mathfrak{u}(n),
\]
and hence, on $\mathfrak{g}_{p,q}=\mathfrak{u}(p+1)\oplus\mathfrak{u}(q+1)$,
\[
B=B_{\mathfrak{u}(p+1)}+B_{\mathfrak{u}(q+1)}.
\]
Using \eqref{inner:product:model:A}, we obtain
\[
-B\big{|}_{\mathfrak{m}_i\times\mathfrak{m}_i}=b_i(\cdot,\cdot)\big{|}_{\mathfrak{m}_i\times\mathfrak{m}_i},\ i=1,2,3,
\]
with
\[
b_1=2(p+q),\ b_2=4(p+1),\ b_3=4(q+1).
\]

For a $G_{p,q}$-invariant metric
\[
g=dt^2+f_1(t)^2\,(\cdot,\cdot)\big|_{\mathfrak{m}_1\times\mathfrak{m}_1}
+f_2(t)^2\,(\cdot,\cdot)\big|_{\mathfrak{m}_2\times\mathfrak{m}_2}
+f_3(t)^2\,(\cdot,\cdot)\big|_{\mathfrak{m}_3\times\mathfrak{m}_3},
\]
equations \eqref{Einstein2}–\eqref{Einstein3} take the form
\begin{align}
&\frac{f_1''}{f_1}
-\frac{p}{2}\,\frac{f_1^{2}}{f_2^{4}}
-\frac{q}{2}\,\frac{f_1^{2}}{f_3^{4}}
+2p\,\frac{f_1' f_2'}{f_1 f_2}
+2q\,\frac{f_1' f_3'}{f_1 f_3}
=-\lambda,\label{Einstein:f_1:model:A}\\[4pt]
&\frac{f_2''}{f_2}
-\frac{2(p+1)}{f_2^{2}}
+\frac{1}{2}\,\frac{f_1^{2}}{f_2^{4}}
+(2p-1)\frac{f_2'^2}{f_2^{2}}
+\frac{f_1' f_2'}{f_1 f_2}
+2q\,\frac{f_2' f_3'}{f_2 f_3}
=-\lambda,\label{Einstein:f_2:model:A}\\[4pt]
&\frac{f_3''}{f_3}
-\frac{2(q+1)}{f_3^{2}}
+\frac{1}{2}\,\frac{f_1^{2}}{f_3^{4}}
+(2q-1)\frac{f_3'^2}{f_3^{2}}
+\frac{f_1' f_3'}{f_1 f_3}
+2p\,\frac{f_2' f_3'}{f_2 f_3}
=-\lambda.\label{Einstein:f_3:model:A}
\end{align}
Moreover, by \eqref{differential:psi_0:model:A}, the inner product
\[
4(\cdot,\cdot)\big|_{\mathfrak{m}_1\times\mathfrak{m}_1}
+(\cdot,\cdot)\big|_{\mathfrak{m}_3\times\mathfrak{m}_3}
\]
coincides with the standard metric of curvature one on
$\mathfrak{p}_0=\mathfrak{m}_1\oplus\mathfrak{m}_3\cong T_H(K_0/H)\cong T_{\psi_0(H)}\mathbb{S}^{2q+1}$.
Therefore, by \eqref{general:initial:conditions}, the metric $g$ extends smoothly to $Q_0$ only if
\begin{align}\label{initial:conditions:model:A}
(f_1(0),f_2(0),f_3(0),f_1'(0),f_2'(0),f_3'(0))
=(0,\zeta_0,0,2,0,1),\ f_1''(0)=f_3''(0)=0,
\end{align}
where $\zeta_0>0$ is arbitrary.\\

Similarly, by \eqref{differential:psi_1:model:A}, the inner product
\[
4(\cdot,\cdot)\big|_{\mathfrak{m}_1\times\mathfrak{m}_1}
+(\cdot,\cdot)\big|_{\mathfrak{m}_2\times\mathfrak{m}_2}
\]
coincides with the standard metric of curvature one on
$\mathfrak{p}_1=\mathfrak{m}_1\oplus\mathfrak{m}_2\cong T_H(K_1/H)\cong T_{\psi_1(H)}\mathbb{S}^{2p+1}$.
Reparametrizing by $t\mapsto 1-t$ and applying again \eqref{general:initial:conditions}, we conclude that
$g$ extends smoothly to $Q_1$ only if
\begin{align}\label{final:conditions:model:A}
(f_1(1),f_2(1),f_3(1),f_1'(1),f_2'(1),f_3'(1))
=(0,0,\zeta_1,-2,-1,0),\ f_1''(1)=f_2''(1)=0,
\end{align}
where $\zeta_1>0$ is arbitrary.

\begin{proposition}\label{proposition:local:solutions:model:A} The following statements hold:
\begin{itemize}
    \item[$a)$] The system \eqref{Einstein:f_1:model:A}-\eqref{Einstein:f_3:model:A}, subject to the initial conditions \eqref{initial:conditions:model:A}, admits infinitely many solutions defined on intervals of the form $[0,\epsilon_0)$ for some $\epsilon_0>0.$
    \item[$b)$] The system \eqref{Einstein:f_1:model:A}-\eqref{Einstein:f_3:model:A}, subject to the final conditions \eqref{final:conditions:model:A}, admits infinitely many solutions defined on intervals of the form $(1-\epsilon_1,1]$ for some $\epsilon_1>0.$
    \item[$c)$] A necessary condition for the existence of a smooth, globally defined $G_{p,q}$-invariant Einstein metric on $\mathbb{C}P^{p+q+1}$ satisfying both \eqref{initial:conditions:model:A} and \eqref{final:conditions:model:A} is
    \begin{equation}\label{compatibility:condition:model:A}
    \frac{2(p+1)}{\zeta_0^2}-\frac{2(q+1)}{\zeta_0}f_2''(0)=\frac{2(q+1)}{\zeta_1^2}-\frac{2(p+1)}{\zeta_1}f_3''(1)=\lambda.
    \end{equation}
\end{itemize}
\end{proposition}
\begin{proof} Observe that, by rescaling $f_1$ to $\frac{f_1}{2},$ we can rewrite \eqref{Einstein:f_1:model:A}-\eqref{final:conditions:model:A} in the form
\begin{align}
&\frac{f_1''}{f_1}
-2p\,\frac{f_1^{2}}{f_2^{4}}
-2q\,\frac{f_1^{2}}{f_3^{4}}
+2p\,\frac{f_1' f_2'}{f_1 f_2}
+2q\,\frac{f_1' f_3'}{f_1 f_3}
=-\lambda,\label{scaled:Einstein:f_1:model:A}\\[4pt]
&\frac{f_2''}{f_2}
-\frac{2(p+1)}{f_2^{2}}
+2\,\frac{f_1^{2}}{f_2^{4}}
+(2p-1)\frac{f_2'^2}{f_2^{2}}
+\frac{f_1' f_2'}{f_1 f_2}
+2q\,\frac{f_2' f_3'}{f_2 f_3}
=-\lambda,\label{scaled:Einstein:f_2:model:A}\\[4pt]
&\frac{f_3''}{f_3}
-\frac{2(q+1)}{f_3^{2}}
+2\,\frac{f_1^{2}}{f_3^{4}}
+(2q-1)\frac{f_3'^2}{f_3^{2}}
+\frac{f_1' f_3'}{f_1 f_3}
+2p\,\frac{f_2' f_3'}{f_2 f_3}
=-\lambda,\label{scaled:Einstein:f_3:model:A}\\
&(f_1(0),f_2(0),f_3(0),f_1'(0),f_2'(0),f_3'(0))=(0,\zeta_0,0,1,0,1),\ f_1''(0)=f_3''(0)=0,\label{scaled:initial:conditions:model:A}\\
&(f_1(1),f_2(1),f_3(1),f_1'(1),f_2'(1),f_3'(1))=(0,0,\zeta_1,-1,-1,0),\ f_1''(1)=f_2''(1)=0.\label{scaled:final:conditions:model:A}
\end{align}

Following \cite[Section 3]{EW}, consider the change of variables
\begin{equation}\label{change:of:variables:0:model:A}
x_i=\frac{f_i^2}{t^2},\ i=1,3,\  x_2=f_2^2,\ y_i=\frac{f_if_i'}{t^2}-\frac{f_i^2}{t^3},\ i=1,3,\ y_2=f_2f_2'.
\end{equation}
Set $x=(x_1,x_2,x_3)$ and $y=(y_1,y_2,y_3).$ Then, equations \eqref{scaled:Einstein:f_1:model:A}-\eqref{scaled:Einstein:f_3:model:A} are equivalent to 
\begin{align}\label{EW:equation}
    \begin{split}
    &x'=2y,\\
    &y'=\frac{1}{t^2}A(x)+\frac{1}{t}B(x,y)+C(t,x,y),
    \end{split}
\end{align}
where
\begin{align*}
    &A(x)=(A_1(x),A_2(x),A_3(x)),\\
    &B(x,y)=(B_1(x,y),B_2(x,y),B_3(x,y)),\\
    &C(t,x,y)=(C_1(t,x,y),C_2(t,x,y),C_3(t,x,y)),
\end{align*}
and
\begin{align*}
    &A_1(x)=2q\frac{x_1^2}{x_3^2}-2qx_1,\\
    &A_2(x)=0,\\
    &A_3(x)=2(q+1)\;-\;2\,\frac{x_1}{x_3}-2q\,x_3,\\
    &B_1(x,y)=-2p\,\frac{x_1y_2}{x_2}-2q\,\frac{x_1y_3}{x_3}-2(q+1)\,y_1,\\
    &B_2(x,y)=-(2q+1)\,y_2,\\
    &B_3(x,y)=-(4q+1)\,y_3-\frac{x_3y_1}{x_1}-2p\frac{x_3y_2}{x_2},\\
    &C_1(t,x,y)=\frac{y_1^{2}}{x_1}-2p\,\frac{y_1y_2}{x_2}-2q\,\frac{y_1y_3}{x_3}-\lambda x_1+2p\frac{x_1^{2}}{x_2^{2}}\,t^{2},\\
    &C_2(t,x,y)=-2(p-1)\,\frac{y_2^{2}}{x_2}-\frac{y_1y_2}{x_1}-2q\,\frac{y_2y_3}{x_3}-\lambda x_2-2\frac{x_1}{x_2}\,t^{2}+2(p+1),\\
    &C_3(t,x,y)=-2(q-1)\,\frac{y_3^{2}}{x_3}
    -\frac{y_1y_3}{x_1}-2p\,\frac{y_2y_3}{x_2}-\lambda x_3.
\end{align*}
The initial conditions \eqref{scaled:initial:conditions:model:A} become
\begin{equation}\label{new:initial:model:A}
    \bigl(x(0),y(0)\bigr)=\bigl((1,\zeta_0^{2},1),(0,0,0)\bigr).
\end{equation}
Observe that the maps $A,B$ satisfy
\begin{align*}
    &A(1,\zeta_0^{2},1)=(0,0,0),\\
    &2\,(dA)_{(1,\zeta_0^{2},1)}(0,0,0)+B((1,\zeta_0^{2},1),(0,0,0))=(0,0,0),
\end{align*}
where $dA$ denotes the Jacobian of $A$. Since $A$, $B$, and $C$ are analytic in a neighborhood of \eqref{new:initial:model:A}, \cite[Section~5]{EW} yields a power-series solution to \eqref{EW:equation} with initial data \eqref{new:initial:model:A}, namely,
\begin{equation}\label{power:series:solution}
    x(t)=\sum_{m=0}^{\infty}\frac{x^{m}}{m!}\,t^{m},\
    y(t)=\sum_{m=0}^{\infty}\frac{y^{m}}{m!}\,t^{m},
\end{equation}
defined on $[0,\epsilon_0)$ for some $\epsilon_0>0.$ Moreover, $x'=2y$ implies $y^m=\frac{1}{2}x^{m+1}$ for all $m\geq0.$ Write 
\begin{align}\label{power:series:A:B:C}
\begin{split}
&A(x(t))=\sum\limits_{m=0}^{\infty}\frac{A^m}{m!}t^m,\\
&B(x(t),y(t))=\sum\limits_{m=0}^{\infty}\frac{B^m}{m!}t^m,\\ &C(t,x(t),y(t))=\sum\limits_{m=0}^{\infty}\frac{C^m}{m!}t^m.
\end{split}
\end{align}
If
\begin{equation}\label{Lm:definition}
\mathcal{L}_m:=(m+1)\mathrm{Id}-\frac{2}{m+2}\,(dA)_{x(0)}-(\partial_yB)_{(x(0),y(0))},
\end{equation}
where $\partial_yB$ denotes the Jacobian of $B$ with respect to $y$, then for $m\geq 0$ the coefficients  $x^m$ can be obtained by solving the linear equation 
\begin{align*}
\mathcal{L}_m(x^{m+2})=D_m,
\end{align*}
where
\begin{align}\label{Dm:definition}
\begin{split}
D_m:=&\,2(m+1)C^m+\frac{2}{m+2}\left(A^{m+2}-(dA)_{x(0)}(x^{m+2})\right)\\[1em]
&+2\left(B^{m+1}-\frac{1}{2}\left(\partial_y B\right)_{(x(0),y(0))}(x^{m+2})\right),
\end{split}\end{align}
depends only on $x^1,...,x^{m+1}.$ Since 
\[
\det(\mathcal{L}_m)=\frac{m(m+3)(m+2q+2)(m+2q+4)(m+4q+2)}{(m+2)^2}>0,\ \forall\, m\ge 1,
\]
the operator $\mathcal{L}_m$ is invertible for all $m\ge 1.$ Therefore, once $x^{0},x^{1},x^{2}$ are prescribed, the coefficients $x^{m},y^{m}$ in \eqref{power:series:solution} are uniquely determined for $m\ge 3$. Note that $x^{0}=x(0)=(1,\zeta_0^{2},1),$ $x^{1}=x'(0)=2y(0)=(0,0,0),$ whereas $x^2$ may be any vector satisfying
\begin{equation*}
\mathcal{L}_0(x^2)=D_0=(-2\lambda,4(p+1)-2\lambda\zeta_0^2,-2\lambda).
\end{equation*}
Solving this equation yields 
\[x^2=\left(-\frac{2\lambda}{3},\frac{2(p+1)-\lambda\zeta_0^2}{q+1},-\frac{p(2(p+1)-\lambda\zeta_0^2)}{3q\zeta_0^2(q+1)}\right)+s\left(-2q,0,1\right),\ s\in\mathbb{R}.\]
Hence, each $s\in\mathbb{R}$ determines a distinct power-series solution \eqref{power:series:solution}. This proves $a).$\\

For $b),$ note that \eqref{scaled:Einstein:f_1:model:A}–\eqref{scaled:Einstein:f_3:model:A} are invariant under the reparametrization $g_i(t)=f_i(1-t)$, $i=1,2,3$, and that the final conditions \eqref{scaled:final:conditions:model:A} are equivalent to
\[
(g_1(0),g_2(0),g_3(0),g_1'(0),g_2'(0),g_3'(0))=(0,0,\zeta_1,1,1,0),\
g_1''(0)=g_2''(0)=0.
\]
With the change of variables
\begin{equation*}
    x_i=\frac{g_i^{2}}{t^{2}},\ i=1,2;\ x_3=g_3^{2};\
    y_i=\frac{g_i g_i'}{t^{2}}-\frac{g_i^{2}}{t^{3}},\ i=1,2;\
    y_3=g_3 g_3',
\end{equation*}
the same argument as in $a)$ applies and yields infinitely many solutions defined on intervals of the form $(1-\epsilon_1,1]$ for some $\epsilon_1>0$. This proves $b).$\\

To prove $c)$, assume that there exists a smooth, globally defined
$G_{p,q}$-invariant Einstein metric on $\mathbb{C}P^{p+q+1}$ satisfying
\eqref{initial:conditions:model:A} and \eqref{final:conditions:model:A}.
By \cite[Theorem 5.2]{DK}, such a metric must be real-analytic. Therefore, near $t=0$ the functions $f_1,f_2,f_3$ admit power-series expansions, so the coefficient computation carried out in part $a)$ applies to any such metric. Observe that in the proof of $a)$ we found that the second coordinate of $x^2$ must satisfy
\[
(x^2)_{2}=\frac{2(p+1)-\lambda\zeta_0^2}{q+1},
\]
which provides
\[
\frac{2(p+1)}{\zeta_0^2}-\frac{2(q+1)}{\zeta_0}f_2''(0)=\lambda.
\]
Applying an analogous argument at $t=1,$ we also obtain
\[
\frac{2(q+1)}{\zeta_1^2}-\frac{2(p+1)}{\zeta_1}f_3''(1)=\lambda.
\]
Therefore, a necessary condition on the boundary values $f_2''(0)$ and $f_3''(1)$ for the existence of a smooth, globally defined $G_{p,q}$-invariant Einstein metric is 
\[
\frac{2(p+1)}{\zeta_0^2}-\frac{2(q+1)}{\zeta_0}f_2''(0)=\frac{2(q+1)}{\zeta_1^2}-\frac{2(p+1)}{\zeta_1}f_3''(1)=\lambda.
\]
This proves $c).$
\end{proof}
\begin{remark}\label{remark:L0:automatic:model:A}
Recall from \eqref{extension:conditions} that, in the smooth extension problem at $Q_0$, the endomorphism
$L_0$ is a prescribed $H$-invariant traceless symmetric operator on $\mathfrak{n}_0.$ In this paper we restrict to the case $L_0=0$. In Model {\bf A}, this restriction is in fact automatic. Indeed, $\mathfrak{n}_0=\mathfrak{m}_2$ is $H$-irreducible, so any
$H$-invariant symmetric endomorphism of $\mathfrak{n}_0$ must be a multiple of the identity. Since $L_0$ is traceless,
it follows that $L_0=0$. An analogous argument applies at the other singular orbit $Q_1,$ because $\mathfrak{n}_1=\mathfrak{m}_3$ is also $H$-irreducible. Consequently, any smooth, globally defined $G_{p,q}$-invariant Einstein metric on $\mathbb{C}P^{p+q+1}$ necessarily satisfies
the boundary conditions \eqref{initial:conditions:model:A} and \eqref{final:conditions:model:A} for some $\zeta_0,\zeta_1>0$.
In particular, the condition \eqref{compatibility:condition:model:A} is necessary
for the existence of such a metric, without imposing \eqref{initial:conditions:model:A}-\eqref{final:conditions:model:A} a priori.
\end{remark}

\begin{remark}\label{remark:FS:model:A}
The Fubini-Study metric is a smooth globally defined solution of the Einstein equations in Model {\bf A}. With the normalization $t\in[0,1]$ used above, a homothetic copy of the Fubini-Study metric is given by
\[
f_1(t)=\frac{2}{\pi}\sin(\pi t),\ f_2(t)=\frac{2}{\pi}\cos\left(\frac{\pi t}{2}\right),\ f_3(t)=\frac{2}{\pi}\sin\left(\frac{\pi t}{2}\right).
\]
These functions satisfy \eqref{Einstein:f_1:model:A}-\eqref{Einstein:f_3:model:A} with $\lambda=\frac{\pi^2}{2}(p+q+2),$ as well as the boundary conditions \eqref{initial:conditions:model:A} and \eqref{final:conditions:model:A}, with $\zeta_0=\zeta_1=\frac{2}{\pi}.$ Moreover,
\[
f_2''(0)=f_3''(1)=-\frac{\pi}{2},
\]
so both expressions in \eqref{compatibility:condition:model:A} are equal to $\lambda$. This solution is homogeneous, so the existence of inhomogeneous smooth globally defined $G_{p,q}$-invariant Einstein metrics remains open.
\end{remark}

\subsection{Model B}
Let $\mathfrak{m}_1,\ldots,\mathfrak{m}_5$ be as in \eqref{reductive:decomposition:explicit:model:B} and $(\cdot,\cdot)$ as in \eqref{inner:product:model:B}. Set
\begin{align*}
    &\mathcal{B}_1:=\{X_{11}\},\
    \mathcal{B}_2:=\{X_{21},X_{22}\},\ \mathcal{B}_3:=\{X_{31},X_{32}\},\\
    &\mathcal{B}_4:=\{X_{4j},Y_{4j}:1\le j\le n-1\},\ \mathcal{B}_5:=\{X_{5j},Y_{5j}:1\le j\le n-1\}.
\end{align*}
Then $\mathcal{B}:=\mathcal{B}_1\cup\cdots\cup\mathcal{B}_5$ is an $(\cdot,\cdot)$-orthonormal basis adapted to
\(\mathfrak{m}_1\oplus\cdots\oplus\mathfrak{m}_5.\) By \eqref{inner:product:model:B}, the constants $b_i$ in \eqref{equation:b_i} are
\[
b_i=1,\ i=1,\ldots,5.
\]

The only nonzero values of $[ijk]$ (up to permutations of $(i,j,k)$) are
\[
[123]=\frac{2}{n+3},\
[133]=\frac{(n-1)^2}{(n+1)(n+3)},\
[144]=[155]=\frac{n-1}{(n+1)(n+3)},
\]
\[
[245]=\frac{n-1}{n+3},\ [345]=\frac{2(n-1)}{(n+1)(n+3)}.
\]

Let $g$ be a $\mathrm{SU}(2)\times\mathrm{SU}(n+1)$-invariant metric of the form
\begin{equation}\label{metric:g:model:B}
    g=dt^2+\sum_{i=1}^{5} f_i(t)^2\,(\cdot,\cdot)\big|_{\mathfrak{m}_i\times\mathfrak{m}_i}.
\end{equation}
Then the equations \eqref{Einstein2}-\eqref{Einstein3} are given by
\begin{align}
&\begin{aligned}\label{Einstein:f1:model:B}
&\frac{f_1''}{f_1}-\frac{2}{n+3}\frac{1}{f_1^2}-\frac{1}{n+3}\left(\frac{f_1^2}{f_2^2f_3^2}-\frac{f_2^2}{f_1^2f_3^2}-\frac{f_3^2}{f_1^2f_2^2}\right)\\[0.5em]
&-\frac{(n-1)}{4(n+1)(n+3)}\left((n-1)\frac{f_1^2}{f_3^4}+\frac{f_1^2}{f_4^4}+\frac{f_1^2}{f_5^4}\right)\\[0.5em]
&+2\frac{f_1'f_2'}{f_1f_2}+2\frac{f_1'f_3'}{f_1f_3}+2(n-1)\frac{f_1'f_4'}{f_1f_4}+2(n-1)\frac{f_1'f_5'}{f_1f_5}=-\lambda,
\end{aligned}\\[1em]
&\begin{aligned}\label{Einstein:f2:model:B}
&\frac{f_2''}{f_2}-\frac{1}{2f_2^2}-\frac{1}{2(n+3)}\left(\frac{f_2^2}{f_1^2f_3^2}-\frac{f_1^2}{f_2^2f_3^2}-\frac{f_3^2}{f_1^2f_2^2}\right)\\[0.5em]
&-\frac{n-1}{4(n+3)}\left(\frac{f_2^2}{f_4^2f_5^2}-\frac{f_4^2}{f_2^2f_5^2}-\frac{f_5^2}{f_2^2f_4^2}\right)\\[0.5em]
&+\frac{f_2'^2}{f_2^2}+\frac{f_1'f_2'}{f_1f_2}+2\frac{f_2'f_3'}{f_2f_3}+2(n-1)\frac{f_2'f_4'}{f_2f_4}+2(n-1)\frac{f_2'f_5'}{f_2f_5}=-\lambda,
\end{aligned}\\[1em]
&\begin{aligned}\label{Einstein:f3:model:B}
    &\frac{f_3''}{f_3}-\frac{1}{2f_3^2}-\frac{1}{2(n+3)}\left(\frac{f_3^2}{f_1^2f_2^2}-\frac{f_1^2}{f_2^2f_3^2}-\frac{f_2^2}{f_1^2f_3^2}\right)\\[0.5em]
    &+\frac{n-1}{2(n+1)(n+3)}\left(\frac{n-1}{2}\frac{f_1^2}{f_3^4}-\frac{f_3^2}{f_4^2f_5^2}+\frac{f_4^2}{f_3^2f_5^2}+\frac{f_5^2}{f_3^2f_4^2}\right)\\[0.5em]
    &+\frac{f_3'^2}{f_3^2}+\frac{f_1'f_3'}{f_1f_3}+2\frac{f_2'f_3'}{f_2f_3}+2(n-1)\frac{f_3'f_4'}{f_3f_4}+2(n-1)\frac{f_3'f_5'}{f_3f_5}=-\lambda,
\end{aligned}\\[1em]
&\begin{aligned}\label{Einstein:f4:model:B}
    &\frac{f_4''}{f_4}-\frac{1}{2f_4^2}-\frac{1}{4(n+3)}\left(\frac{f_4^2}{f_2^2f_5^2}-\frac{f_2^2}{f_4^2f_5^2}-\frac{f_5^2}{f_2^2f_4^2}\right)\\[0.5em]
    &+\frac{1}{2(n+1)(n+3)}\left(\frac{1}{2}\frac{f_1^2}{f_4^4}-\frac{f_4^2}{f_3^2f_5^2}+\frac{f_3^2}{f_4^2f_5^2}+\frac{f_5^2}{f_3^2f_4^2}\right)\\[0.5em]
    &+(2n-3)\frac{f_4'^2}{f_4^2}+\frac{f_1'f_4'}{f_1f_4}+2\frac{f_2'f_4'}{f_2f_4}+2\frac{f_3'f_4'}{f_3f_4}+2(n-1)\frac{f_4'f_5'}{f_4f_5}=-\lambda,
\end{aligned}\\[1em]
&\begin{aligned}\label{Einstein:f5:model:B}
    &\frac{f_5''}{f_5}-\frac{1}{2f_5^2}-\frac{1}{4(n+3)}\left(\frac{f_5^2}{f_2^2f_4^2}-\frac{f_2^2}{f_4^2f_5^2}-\frac{f_4^2}{f_2^2f_5^2}\right)\\[0.5em]
    &+\frac{1}{2(n+1)(n+3)}\left(\frac{1}{2}\frac{f_1^2}{f_5^4}-\frac{f_5^2}{f_3^2f_4^2}+\frac{f_3^2}{f_4^2f_5^2}+\frac{f_4^2}{f_3^2f_5^2}\right)\\[0.5em]
    &+(2n-3)\frac{f_5'^2}{f_5^2}+\frac{f_1'f_5'}{f_1f_5}+2\frac{f_2'f_5'}{f_2f_5}+2\frac{f_3'f_5'}{f_3f_5}+2(n-1)\frac{f_4'f_5'}{f_4f_5}=-\lambda.
\end{aligned}
\end{align}
\begin{align}\label{Einstein:nondiagonal:model:B}
-\frac{f_1^2}{f_3^2}
-\frac{f_2^2f_3^2}{f_4^2f_5^2}
+\frac{f_4^2}{f_5^2}+\frac{f_5^2}{f_4^2}
+\frac{f_2^2}{f_1^2}-\frac{f_3^2}{f_1^2}=0.
\end{align}
To determine the boundary conditions, recall that $T_{H}(K_0/H)$ is identified with
$\mathfrak{p}_0=\mathfrak{m}_1\oplus\mathfrak{m}_4$.
By \eqref{derivatives:psi0:model:B}, the inner product
\[
\displaystyle\frac{n+3}{2(n+1)}\,(\cdot,\cdot)\big|_{\mathfrak{m}_1\times\mathfrak{m}_1}
+\frac{1}{4(n+1)}\,(\cdot,\cdot)\big|_{\mathfrak{m}_4\times\mathfrak{m}_4},
\]
coincides with the standard metric of curvature one on $\mathfrak{p}_0\cong T_{\psi_0(H)}\mathbb{S}^{2n-1}$.
Therefore, by \eqref{general:initial:conditions}, the metric \eqref{metric:g:model:B}
extends smoothly to $t=0$ provided that
\begin{align}\label{initial:conditions:model:B}
\begin{split}&(f_1(0),f_2(0),f_3(0),f_4(0),f_5(0))=(0,\zeta_0,\zeta_0,0,\zeta_0),\ \zeta_0>0,\\[0.2em]
&(f_1'(0),f_2'(0),f_3'(0),f_4'(0),f_5'(0))
=\left(\sqrt{\frac{n+3}{2(n+1)}},\,0,\,0,\,\frac{1}{2\sqrt{n+1}},\,0\right),\\[0.2em]
&f_1''(0)=f_4''(0)=0.
\end{split}
\end{align}

We argue similarly for the other singular orbit. Recall that $T_{H}(K_1/H)$ is identified with
$\mathfrak{p}_1=\mathfrak{m}_2$.
By \eqref{derivatives:psi1:model:B}, the inner product
\[
\frac{1}{n+3}\,(\cdot,\cdot)\big|_{\mathfrak{m}_2\times\mathfrak{m}_2},
\]
coincides with the standard metric of curvature one on $\mathfrak{p}_1\cong T_{\psi_1(H)}\mathbb{S}^{2}$.
Reparametrizing by $t\mapsto 1-t$ and applying again \eqref{general:initial:conditions}, we obtain that
$g$ extends smoothly to $t=1$ provided that
\begin{align}\label{final:conditions:model:B}
\begin{split}&(f_1(1),f_2(1),f_3(1),f_4(1),f_5(1))
=\left(\zeta_1,0,\zeta_1,\xi_1,\xi_1\right),\ \zeta_1>0,\ \xi_1>0,\\[0.2em]
&(f_1'(1),f_2'(1),f_3'(1),f_4'(1),f_5'(1))
=\left(0,\,-\frac{1}{\sqrt{n+3}},\,0,\,0,\,0\right),\ f_2''(1)=0,
\end{split}
\end{align}
where we used that, in this case, the $K_1$-module $\mathfrak{n}_1$ splits into the two inequivalent
irreducible submodules $\mathfrak{m}_1\oplus\mathfrak{m}_3$ and $\mathfrak{m}_4\oplus\mathfrak{m}_5$,
so that the metric operator on $\mathfrak{n}_1$ is determined by two independent parameters $\zeta_1$ and $\xi_1$.
\begin{proposition}\label{local:solutions:model:B}
    The following statements hold:
    \begin{itemize}
        \item[$a)$] The system \eqref{Einstein:f1:model:B}-\eqref{Einstein:f5:model:B}, subject to the initial conditions \eqref{initial:conditions:model:B}, admits infinitely many solutions defined on intervals of the form $[0,\epsilon_0)$ for some $\epsilon_0>0.$
        \item[$b)$] The system \eqref{Einstein:f1:model:B}-\eqref{Einstein:f5:model:B}, subject to the final conditions \eqref{final:conditions:model:B}, admits infinitely many solutions defined on intervals of the form $(1-\epsilon_1,1]$ for some $\epsilon_1>0.$ 
    \end{itemize}
\end{proposition}
\begin{proof} We begin with part $b).$ Observe that equations \eqref{Einstein:f1:model:B}-\eqref{Einstein:f5:model:B} are invariant under the reparametrization $g_i(t)=f_i(1-t),\ i=1,2,3,4,5,$ and that the final conditions \eqref{final:conditions:model:B} become 
\begin{align}\label{reparametrized:final:conditions:model:B}
\begin{split}
&(g_1(0),g_2(0),g_3(0),g_4(0),g_5(0))=(\zeta_1,0,\zeta_1,\xi_1,\xi_1),\\
&(g_1'(0),g_2'(0),g_3'(0),g_4'(0),g_5'(0))=\Bigl(0,\frac{1}{\sqrt{n+3}},0,0,0\Bigr),\ g_2''(0)=0.
\end{split}
\end{align}
As in the proof of Proposition \ref{proposition:local:solutions:model:A}, define the new variables 
\begin{equation*}
    x_2=\frac{g_2^2}{t^2},\ y_2=\frac{g_2g_2'}{t^2}-\frac{g_2^2}{t^3},\ x_i=g_i^2,\ y_i=g_ig_i',\ i=1,3,4,5. 
\end{equation*}
Then the system \eqref{Einstein:f1:model:B}-\eqref{Einstein:f5:model:B} can be rewritten as 
\begin{align*}
    &x'=2y\\
    &y'=\frac{1}{t^2}A(x)+\frac{1}{t}B(x,y)+C(t,x,y),
\end{align*}
where
\begin{align*}
    &x=(x_1,x_2,x_3,x_4,x_5),\ y=(y_1,y_2,y_3,y_4,y_5),\\
    &A(x)=(A_1(x),A_2(x),A_3(x),A_4(x),A_5(x)),\\
    &B(x,y)=(B_1(x,y),B_2(x,y),B_3(x,y),B_4(x,y),B_5(x,y))\\
    &C(t,x,y)=(C_1(t,x,y),C_2(t,x,y),C_3(t,x,y),C_4(t,x,y),C_5(t,x,y))
\end{align*}
and
\begin{align*}
&A_1(x) = \frac{1}{n+3}\left(\frac{x_1^2}{x_2x_3}-\frac{x_3}{x_2}\right),\\[0.5em]
&A_2(x) = -x_2+\frac12-\frac{1}{2(n+3)}\left(\frac{x_1}{x_3}+\frac{x_3}{x_1}\right)
-\frac{n-1}{4(n+3)}\left(\frac{x_4}{x_5}+\frac{x_5}{x_4}\right),\\[0.5em]
&A_3(x) = \frac{1}{2(n+3)}\left(\frac{x_3^2}{x_1x_2}-\frac{x_1}{x_2}\right),\\[0.5em]
&A_4(x) = \frac{1}{4(n+3)}\left(\frac{x_4^2}{x_2x_5}-\frac{x_5}{x_2}\right),\ 
A_5(x)= \frac{1}{4(n+3)}\left(\frac{x_5^2}{x_2x_4}-\frac{x_4}{x_2}\right),\\[0.5em]
&B_1(x,y) = -2y_1,\
B_3(x,y) = -2y_3,\
B_4(x,y) = -2y_4,\
B_5(x,y) = -2y_5,\\[0.5em]
&B_2(x,y) = -4y_2-\frac{x_2y_1}{x_1}-2\frac{x_2y_3}{x_3}
-2(n-1)\frac{x_2y_4}{x_4}-2(n-1)\frac{x_2y_5}{x_5},\\[0.5em]
&\begin{aligned}
    C_1(t,x,y) = &\,\frac{y_1^2}{x_1}-2\frac{y_1y_2}{x_2}-2\frac{y_1y_3}{x_3}-2(n-1)\frac{y_1y_4}{x_4}-2(n-1)\frac{y_1y_5}{x_5}-\lambda x_1\\
&\,+\frac{n-1}{4(n+1)(n+3)}\left((n-1)\frac{x_1^2}{x_3^2}+\frac{x_1^2}{x_4^2}+\frac{x_1^2}{x_5^2}\right)-\frac{x_2}{(n+3)x_3}t^2
+\frac{2}{n+3},
\end{aligned}\\[0.5em]
&\begin{aligned}C_2(t,x,y) = &-\frac{y_1y_2}{x_1}-2\frac{y_2y_3}{x_3}
-2(n-1)\frac{y_2y_4}{x_4}-2(n-1)\frac{y_2y_5}{x_5}-\lambda x_2\\
&+\frac{1}{2(n+3)}\left(\frac{x_2^2}{x_1x_3}+\frac{n-1}{2}\frac{x_2^2}{x_4x_5}\right)t^2,
\end{aligned}\\[0.5em]
&\begin{aligned}C_3(t,x,y) = &-\frac{y_1y_3}{x_1}-2\frac{y_2y_3}{x_2}
-2(n-1)\frac{y_3y_4}{x_4}-2(n-1)\frac{y_3y_5}{x_5}-\lambda x_3\\
&-\frac{n-1}{2(n+1)(n+3)}\left(\frac{n-1}{2}\frac{x_1}{x_3}-\frac{x_3^2}{x_4x_5}+\frac{x_4}{x_5}+\frac{x_5}{x_4}\right)-\frac{t^2x_2}{2(n+3)x_1}+\frac{1}{2},
\end{aligned}\\[0.5em]
&\begin{aligned}C_4(t,x,y) = &-2(n-2)\frac{y_4^2}{x_4}-\frac{y_1y_4}{x_1}-2\frac{y_2y_4}{x_2}-2\frac{y_3y_4}{x_3}
-2(n-1)\frac{y_4y_5}{x_5}-\lambda x_4\\
&-\frac{1}{2(n+1)(n+3)}\left(\frac12\frac{x_1}{x_4}-\frac{x_4^2}{x_3x_5}+\frac{x_3}{x_5}+\frac{x_5}{x_3}\right)-\frac{x_2}{4(n+3)x_5}t^2+\frac{1}{2},
\end{aligned}\\[0.5em]
&\begin{aligned}C_5(t,x,y) = &-2(n-2)\frac{y_5^2}{x_5}-\frac{y_1y_5}{x_1}-2\frac{y_2y_5}{x_2}-2\frac{y_3y_5}{x_3}
-2(n-1)\frac{y_4y_5}{x_4}-\lambda x_5\\
&-\frac{1}{2(n+1)(n+3)}\left(\frac12\frac{x_1}{x_5}-\frac{x_5^2}{x_3x_4}+\frac{x_3}{x_4}+\frac{x_4}{x_3}\right)-\frac{x_2}{4(n+3)x_4}t^2+\frac{1}{2}.
\end{aligned}
\end{align*}
In these variables, the conditions \eqref{reparametrized:final:conditions:model:B} are given by
\begin{equation*}
    x(0)=\left(\zeta_1^2,\frac{1}{n+3},\zeta_1^2,\xi_1^2,\xi_1^2\right),\ y(0)=(0,0,0,0,0).
\end{equation*}
A straightforward computation shows that   
\begin{align*}
    A(x(0))=2(dA)_{x(0)}(y(0))+B(x(0),y(0))=(0,0,0,0,0).
\end{align*}
Consider a power-series solution
\begin{equation*}
    x(t)=\sum\limits_{m=0}^{\infty}\frac{x^m}{m!}t^m,\ y(t)=\sum\limits_{m=0}^{\infty}\frac{y^m}{m!}t^m.
\end{equation*}
and write
\begin{align*}
\begin{split}
A(x(t))=\sum\limits_{m=0}^{\infty}\frac{A^m}{m!}t^m,\ B(x(t),y(t))=\sum\limits_{m=0}^{\infty}\frac{B^m}{m!}t^m,\ C(t,x(t),y(t))=\sum\limits_{m=0}^{\infty}\frac{C^m}{m!}t^m.
\end{split}
\end{align*}
Recall that
\[
\mathcal{L}_m=(m+1)\mathrm{Id}-\frac{2}{m+2}(dA)_{x(0)}-(\partial_yB)_{(x(0),y(0))}
\]
and
\[
\begin{aligned}
D_m=&\,2(m+1)C^m+\frac{2}{m+2}\left(A^{m+2}-(dA)_{x(0)}\cdot x^{m+2}\right)\\
&+2\left(B^{m+1}-\frac12(\partial_yB)_{(x(0),y(0))}( x^{m+2})\right).
\end{aligned}
\]
Then
\[
\det(\mathcal{L}_m)=\frac{m(m+1)(m+3)^3(m+4)^2(m+5)}{(m+2)^3},
\]
so \(\mathcal{L}_m\) is invertible for every \(m\geq 1\). As in the proof of Proposition \ref{proposition:local:solutions:model:A}, it follows that once $x^0,x^1,x^2$ are  fixed, all higher coefficients are uniquely determined by the equation 
\[
\mathcal{L}_m(x^{m+2})=D_m\Longleftrightarrow x^{m+2}=\mathcal{L}_m^{-1}(D_m),\ m\geq 1.
\] Here \[
x^0=x(0)=\left(\zeta_1^2,1/(n+3),\zeta_1^2,\xi_1^2,\xi_1^2\right),\ x^1=2y(0)=(0,0,0,0,0)
\]
and $x^2$ is any solution of the linear equation
\[\mathcal{L}_0(x^2)=D_0=((D_0)_1,(D_0)_2,(D_0)_3,(D_0)_4,(D_0)_5),\]
where 
\begin{align*}
    &(D_0)_1=(D_0)_3=-2\lambda\zeta_1^2+\frac{n+3}{2(n+1)}+\frac{n-1}{(n+1)(n+3)}\frac{\zeta_1^4}{\xi_1^4},\\
    &(D_0)_2=-\frac{2\lambda}{n+3},\\
    &(D_0)_4=(D_0)_5=-2\lambda\xi_1^2+1-\frac{3}{2(n+1)(n+3)}\frac{\zeta_1^2}{\xi_1^2}.
\end{align*}
Solving this equation, we obtain 
\begin{align*}
x^2=&\left((D_0)_1,\frac{1}{6}\left((D_0)_2-\frac{1}{(n+3)\zeta_1^2}(D_0)_1-\frac{4(n-1)}{3(n+3)\xi_1^2}(D_0)_4\right),0,\frac{(D_0)_4}{3},\frac{(D_0)_4}{3}\right)\\[0.3em]
&+s(-2,0,1,0,0),\ s\in\mathbb{R}.
\end{align*}
Thus, there are infinitely many local solutions. This proves $b).$\\

The proof of $a)$ is analogous and we omit it.
\end{proof}
Although the system \eqref{Einstein:f1:model:B}-\eqref{Einstein:f5:model:B} admits infinitely many local solutions near $t=0$ and $t=1$, such solutions are not necessarily Einstein metrics, since in Model {\bf B} the Einstein equations also include the condition \eqref{Einstein:nondiagonal:model:B}. This condition appears because the $H$-modules $\mathfrak{m}_2$ and $\mathfrak{m}_3$ are equivalent. Nevertheless, the analysis above will be used in the next proposition to show that Model {\bf B} does not admit smooth, globally defined $\mathrm{SU}(2)\times \mathrm{SU}(n+1)$-invariant Einstein metrics on $\mathbb{C}P^{2n+1}$ with totally geodesic singular orbits.
\begin{proposition}\label{nonexistence:model:B}
    There is no smooth, globally defined, $\mathrm{SU}(2)\times\mathrm{SU}(n+1)$-invariant Einstein metric\[g=dt^2+\sum\limits_{i=1}^5f_i(t)^2(\cdot,\cdot)\big{|}_{\mathfrak{m}_i\times\mathfrak{m}_i}\] on $\mathbb{C}P^{2n+1}$ with totally geodesic singular orbits.
\end{proposition}
\begin{proof}
Assume that such a metric $g$ exists, then the functions $f_i$ satisfy \eqref{Einstein:f1:model:B}-\eqref{Einstein:nondiagonal:model:B} and the boundary conditions \eqref{initial:conditions:model:B},\eqref{final:conditions:model:B}. Let
\[
g_i(t)=f_i(1-t),\ i=1,\dots,5
\]
and let $x_i,y_i$, $i=1,\dots,5$, be the variables introduced in the proof of Proposition \ref{local:solutions:model:B}. We will prove that
\[
x_4=x_5.
\]

Since $g$ is Einstein, it is real analytic. Hence each $f_i$, and therefore each $x_i$ and $y_i$, admits a power-series expansion on $[0,\epsilon)$ for some $\epsilon>0$. Let $x^m,y^m,A^m,B^m,C^m$ ($m\geq 0$) be as in the proof of the previous proposition, and define $P:\mathbb{R}^5\to\mathbb{R}^5$ by
\begin{equation}\label{map:p}
P(x_1,x_2,x_3,x_4,x_5)=(x_1,x_2,x_3,x_5,x_4).
\end{equation}
A direct computation shows that
\begin{equation}\label{P:invariance:A:B:C}
A(Px)=P(A(x)),\ B(Px,Py)=P(B(x,y)),\ C(t,Px,Py)=P(C(t,x,y)).
\end{equation}
Moreover,
\[
P(x(0))=x(0),\ P(y(0))=y(0).
\]
Therefore,
\[
(dA)_{x(0)}P=P(dA)_{x(0)},\
(\partial_yB)_{(x(0),y(0))}P=P(\partial_yB)_{(x(0),y(0))},
\]
and consequently
\[
P\mathcal{L}_m=\mathcal{L}_mP,\ \forall m\geq 0.
\]

We will prove by induction that
\begin{equation}\label{induction}
P(D_m)=D_m
\ \text{and}\ P(x^{m+2})=x^{m+2},
\ \forall m\geq 0.
\end{equation}

For $m=0$, we already know from the proof of the previous proposition that
\[
(D_0)_4=(D_0)_5=-2\lambda\xi_1^2+1-\frac{3}{2(n+1)(n+3)}\frac{\zeta_1^2}{\xi_1^2},
\]
and
\[
(x^2)_4=(x^2)_5=\frac{(D_0)_4}{3}.
\]
Hence $P(D_0)=D_0$ and $P(x^2)=x^2$. Now let $m\geq 1$ and assume that
\[
P(D_i)=D_i,\ P(x^{i+2})=x^{i+2},\ 0\leq i\leq m-1.
\]
Since
\[
P(x^0)=P(x(0))=x(0)=x^0
\]
and
\[
P(x^1)=P(2y^0)=2P(y(0))=2y(0)=x^1,
\]
it follows that
\[
P(x^k)=x^k,\ 0\leq k\leq m+1.
\]
Therefore,
\[
P(x(t))-x(t)=\sum_{k=m+2}^{\infty}\frac{P(x^k)-x^k}{k!}t^k=O(t^{m+2}),
\]
and, since $y^k=\frac12x^{k+1}$,
\[
P(y(t))-y(t)=\sum_{k=m+1}^{\infty}\frac{P(y^k)-y^k}{k!}t^k
=\sum_{k=m+1}^{\infty}\frac{P(x^{k+1})-x^{k+1}}{2k!}t^k
=O(t^{m+1}).
\]

Let $U$, $V$, and $W$ be convex neighborhoods of $x(0)$, $(x(0),y(0))$, and $(0,x(0),y(0))$, respectively, on which the derivatives of $A$, $B$, and $C$ are bounded. Set
\[
M_A:=\sup_{x\in U} ||(dA)_x||,\ M_B:=\sup_{(x,y)\in V} ||(dB)_{(x,y)}||,\ M_C:=\sup_{(s,x,y)\in W} ||(dC)_{(s,x,y)}||.
\]
Then, for $t$ small enough, the mean value inequality gives
\begin{align*}
&||A(P(x(t)))-A(x(t))||\leq M_A\, ||P(x(t))-x(t)||=O(t^{m+2}),\\[0.2em]
&\begin{aligned}
||B(P(x(t)),P(y(t)))-B(x(t),y(t))||&\leq M_B\bigl(||P(x(t))-x(t)||+||P(y(t))-y(t)||\bigr)\\[0.2em]
&=O(t^{m+1}),
\end{aligned}\\[0.2em]
&\begin{aligned}
||C(t,P(x(t)),P(y(t)))-C(t,x(t),y(t))||&\leq M_C\bigl(||P(x(t))-x(t)||+||P(y(t))-y(t)||\bigr)\\[0.2em]
&=O(t^{m+1}).
\end{aligned}
\end{align*}
Using \eqref{P:invariance:A:B:C}, we obtain
\[
P(A^k)=A^k,\ k=0,\dots,m+1,\ P(B^j)=B^j,\ P(C^j)=C^j,\ j=0,\dots,m.
\]

Next,
\begin{align*}
A^{m+2}
=&\left.\frac{d^{m+2}}{dt^{m+2}}A(x(t))\right|_{t=0}\\
=&\sum_{k=1}^{m+2}\sum_{\vec{r}\in J^k_{m+2}}
\frac{(m+2)!}{k!r_1!\cdots r_k!}
(d^kA)_{x(0)}(x^{r_1},\dots,x^{r_k})\\
=&(dA)_{x(0)}(x^{m+2})\\
&+\sum_{k=2}^{m+2}\sum_{\vec{r}\in J^k_{m+2}}
\frac{(m+2)!}{k!r_1!\cdots r_k!}
(d^kA)_{x(0)}(x^{r_1},\dots,x^{r_k}),
\end{align*}
where
\[
J^k_{m+2}:=\{\vec{r}=(r_1,\dots,r_k)\in\mathbb{N}^k:r_1+\cdots+r_k=m+2\},
\]
and
\[
\begin{aligned}
(d^kA)_{x(0)}(w_1,\dots,w_k)
:={}&\left.\frac{\partial^k}{\partial t_1\cdots\partial t_k}
A(x(0)+t_1w_1+\cdots+t_kw_k)\right|_{(t_1,\dots,t_k)=(0,\dots,0)}.
\end{aligned}
\]
By the linearity of \(P\) and the identities \(A(Px)=P(A(x)),\ P(x(0))=x(0),\) we have
\[
\begin{aligned}
P(d^kA)_{x(0)}(w_1,\dots,w_k)
=(d^kA)_{x(0)}(P(w_1),\dots,P(w_k))
\end{aligned}
\]
for every \(w_1,\dots,w_k\in\mathbb{R}^5\). Moreover, \(P(x^a)=x^a\) for every \(0\leq a\leq m+1\). Therefore each term in the last sum is fixed by \(P\), and so
\[
P\bigl(A^{m+2}-(dA)_{x(0)}(x^{m+2})\bigr)
=
A^{m+2}-(dA)_{x(0)}(x^{m+2}).
\]

A similar argument applies to \(B^{m+1}\). In fact,
\begin{align*}
B^{m+1}
=&\left.\frac{d^{m+1}}{dt^{m+1}}B(x(t),y(t))\right|_{t=0}\\
=&\sum_{\begin{subarray}{c}p,q\geq 0\\1\leq p+q\leq m+1\end{subarray}}
\sum_{(\vec{r},\vec{s})\in J^{p,q}_{m+1}}
\frac{(m+1)!}{p!q!r_1!\cdots r_p!s_1!\cdots s_q!}\\
&\quad\cdot
(d^{p,q}B)_{(x(0),y(0))}(x^{r_1},\dots,x^{r_p},y^{s_1},\dots,y^{s_q})\\
=&(\partial_yB)_{(x(0),y(0))}(y^{m+1})\\
&+\sum_{\begin{subarray}{c}p,q\geq 0\\1\leq p+q\leq m+1\\(p,q)\neq (0,1)\end{subarray}}
\sum_{(\vec{r},\vec{s})\in J^{p,q}_{m+1}}
\frac{(m+1)!}{p!q!r_1!\cdots r_p!s_1!\cdots s_q!}\\
&\qquad\cdot
(d^{p,q}B)_{(x(0),y(0))}(x^{r_1},\dots,x^{r_p},y^{s_1},\dots,y^{s_q})\\
=&\frac12(\partial_yB)_{(x(0),y(0))}(x^{m+2})\\
&+\sum_{\begin{subarray}{c}p,q\geq 0\\1\leq p+q\leq m+1\\(p,q)\neq (0,1)\end{subarray}}
\sum_{(\vec{r},\vec{s})\in J^{p,q}_{m+1}}
\frac{(m+1)!}{p!q!r_1!\cdots r_p!s_1!\cdots s_q!2^q}\\
&\qquad\cdot
(d^{p,q}B)_{(x(0),y(0))}(x^{r_1},\dots,x^{r_p},x^{s_1+1},\dots,x^{s_q+1}),
\end{align*}
where
\[
J^{p,q}_{m+1}
=
\left\{
(\vec{r},\vec{s})=(r_1,\dots,r_p,s_1,\dots,s_q)\in\mathbb{N}^{p+q}
:
r_1+\cdots+r_p+s_1+\cdots+s_q=m+1
\right\},
\]
and

\[
\begin{aligned}
&(d^{p,q}B)_{(x(0),y(0))}(w_1,\dots,w_p,z_1,\dots,z_q)\\
&\qquad:=\left.
\frac{\partial^{p+q}}{\partial u_1\cdots\partial u_p\,\partial v_1\cdots\partial v_q}
B\left(
x(0)+\sum_{i=1}^pu_iw_i,\,
y(0)+\sum_{j=1}^qv_jz_j
\right)\right|_{(\vec{u},\vec{v})=(\vec{0},\vec{0})}.
\end{aligned}
\]
Using again the linearity of \(P\) and the identities \(B(Px,Py)=P(B(x,y)),\) we obtain
\[
\begin{aligned}
&P(d^{p,q}B)_{(x(0),y(0))}(w_1,\dots,w_p,z_1,\dots,z_q)\\
&\qquad=(d^{p,q}B)_{(x(0),y(0))}(P(w_1),\dots,P(w_p),P(z_1),\dots,P(z_q))
\end{aligned}
\]
for every \(w_1,\dots,w_p,z_1,\dots,z_q\in\mathbb{R}^5\). Since \(P(x^a)=x^a\) for every \(0\leq a\leq m+1\), it follows that
\[
P\left(B^{m+1}-\frac12(\partial_yB)_{(x(0),y(0))}(x^{m+2})\right)
=
B^{m+1}-\frac12(\partial_yB)_{(x(0),y(0))}(x^{m+2}).
\]

Combining the previous identities we obtain
\[
P(D_m)=D_m.
\]
Since \(m\geq 1\), the operator \(\mathcal{L}_m\) is invertible and since \(P\mathcal{L}_m=\mathcal{L}_mP,\) then
\[
P\mathcal{L}_m^{-1}=\mathcal{L}_m^{-1}P.
\]
Hence
\[
P(x^{m+2})
=
P(\mathcal{L}_m^{-1}(D_m))
=
\mathcal{L}_m^{-1}(P(D_m))
=
\mathcal{L}_m^{-1}(D_m)
=
x^{m+2}.
\]
This proves \eqref{induction}. Thus \(P(x^m)=x^m\) for every \(m\geq 0\). It follows that \(P(x(t))=x(t)\) for \(t\) in a neighborhood of \(0\). In particular, \(x_4(t)=x_5(t)\)
for \(t\) near \(0\). Since \(x_4-x_5\) is real-analytic on \((0,1)\), it must vanish identically on \((0,1)\). Hence \(x_4(t)=x_5(t),\ \forall t\in(0,1).\) By continuity, \(x_4(1)=x_5(1),\)
but
\[
x_4(1)=g_4(1)^2=f_4(0)^2=0,
\ x_5(1)=g_5(1)^2=f_5(0)^2=\zeta_0^2>0.
\]
This contradiction completes the proof.
\end{proof}
\begin{remark}\label{remark:only:Q1:model:B}
The assumption that both singular orbits are totally geodesic in the proposition above is not needed. In fact, the argument only uses the boundary conditions at \(Q_1\) (coming from the assumption that \(Q_1\) is totally geodesic) to obtain \(x_4=x_5\) on the whole interval. The contradiction arises from the smoothness conditions at \(Q_0\), namely \(f_4(0)=0\ \text{and}\ f_5(0)=\zeta_0>0,\) which hold for any smooth, globally defined $\mathrm{SU}(2)\times\mathrm{SU}(n+1)$-invariant metric. Therefore, the same proof shows that there is no smooth, globally defined, \(\mathrm{SU}(2)\times\mathrm{SU}(n+1)\)-invariant Einstein metric on \(\mathbb{C}P^{2n+1}\) for which \(Q_1\) is totally geodesic.
\end{remark}
\subsection{Model C}\label{Einstein:metrics:C}
Consider the spaces $\mathfrak{m}_1,\mathfrak{m}_2,\mathfrak{m}_3$ defined in \eqref{summands:model:C} and let $(\cdot,\cdot)$ be the inner product defined in \eqref{inner:product:model:C}. An $(\cdot,\cdot)$-orthonormal basis adapted to $\mathfrak{m}_1\oplus\mathfrak{m}_2\oplus\mathfrak{m}_3$ is $\mathcal{B}=\mathcal{B}_1\cup\mathcal{B}_2\cup\mathcal{B}_3$, where
\begin{align*}
    &\mathcal{B}_1=\left\{\frac{1}{2}\left(E_{n+1,n}-E_{n,n+1}\right)\right\},\\
    &\mathcal{B}_2=\left\{\frac{1}{2}\left(E_{nj}-E_{jn}\right):j=1,\ldots,n-1\right\},\\
    &\mathcal{B}_3=\left\{\frac{1}{2}\left(E_{n+1,j}-E_{j,n+1}\right):j=1,\ldots,n-1\right\}.
\end{align*}
For this basis,
\[
    [ijk]=\left\{\begin{array}{ll}
        \displaystyle \frac{n-1}{4}, & \textnormal{if}\ \{i,j,k\}=\{1,2,3\},\\[0.4em]
        0, & \textnormal{otherwise}.
    \end{array}\right.
\]
Since the Killing form of $\mathfrak{so}(n+1)$ is $B(X,Y)=(n-1)\mathrm{Tr}(XY)$, it follows that the constants $b_i$ in \eqref{equation:b_i} are 
\[
b_i=\frac{n-1}{2},\ i=1,2,3.
\]

Now, consider an $\mathrm{SO}(n+1)$-invariant metric
\begin{equation}\label{metric:g:model:C}
g \;=\; dt^2
  + f_1(t)^2\,(\cdot,\cdot)\big|_{\mathfrak{m}_1\times\mathfrak{m}_1}
  + f_2(t)^2\,(\cdot,\cdot)\big|_{\mathfrak{m}_2\times\mathfrak{m}_2}
  + f_3(t)^2\,(\cdot,\cdot)\big|_{\mathfrak{m}_3\times\mathfrak{m}_3}
\end{equation}
on $\mathbb{C}P^n\setminus\{Q_0,Q_1\}$, where $Q_0,Q_1$ are given by \eqref{singular:orbits:model:C}. A direct computation shows that $g$ satisfies \eqref{Einstein2}-\eqref{Einstein3} if and only if
\begin{align}
    &\frac{f_1''}{f_1}-\frac{n-1}{4f_1^2}-\frac{n-1}{8}\left(\frac{f_1^2}{f_2^2f_3^2}-\frac{f_2^2}{f_1^2f_3^2}-\frac{f_3^2}{f_1^2f_2^2}\right)
    +(n-1)\frac{f_1'f_2'}{f_1f_2}+(n-1)\frac{f_1'f_3'}{f_1f_3}=-\lambda, \label{Einstein:CP^n:f_1}\\[0.8em]
    &\frac{f_2''}{f_2}-\frac{n-1}{4f_2^2}-\frac{1}{8}\left(\frac{f_2^2}{f_1^2f_3^2}-\frac{f_1^2}{f_2^2f_3^2}-\frac{f_3^2}{f_1^2f_2^2}\right)
    +(n-2)\frac{f_2'^2}{f_2^2}+(n-1)\frac{f_2'f_3'}{f_2f_3}+\frac{f_1'f_2'}{f_1f_2}=-\lambda, \label{Einstein:CP^n:f_2}\\[0.8em]
    &\frac{f_3''}{f_3}-\frac{n-1}{4f_3^2}-\frac{1}{8}\left(\frac{f_3^2}{f_1^2f_2^2}-\frac{f_1^2}{f_2^2f_3^2}-\frac{f_2^2}{f_1^2f_3^2}\right)
    +(n-2)\frac{f_3'^2}{f_3^2}+(n-1)\frac{f_2'f_3'}{f_2f_3}+\frac{f_1'f_3'}{f_1f_3}=-\lambda. \label{Einstein:CP^n:f_3}
\end{align}
Using \eqref{general:initial:conditions} together with \eqref{derivative:psi0:model:C}, we note that the restriction
$(\cdot,\cdot)\big|_{\mathfrak{m}_1\times\mathfrak{m}_1}$ is already normalized so that
$\mathfrak{p}_0=\mathfrak{m}_1\cong T_{\psi_0(H)}\mathbb{S}^{1}$ carries the standard metric of curvature one. Therefore,
$g$ extends smoothly to $Q_0$ provided that
\begin{align}
\big(f_1(0),f_2(0),f_3(0),f_1'(0),f_2'(0),f_3'(0)\big)
  \;=\; (0,\zeta_0,\zeta_0,1,0,0),\ f_1''(0)=0, \label{initial:values:0:model:C}
\end{align}
where $\zeta_0>0$ is arbitrary.\\

At $t=1$ the situation is different, since $(\cdot,\cdot)$ does not induce the curvature-one metric on
$\mathfrak{p}_1=\mathfrak{m}_2$. By \eqref{derivative:psi1:model:C}, the rescaled inner product
$\frac14(\cdot,\cdot)\big|_{\mathfrak{m}_2\times\mathfrak{m}_2}$ coincides with the standard metric of curvature one on
$\mathfrak{p}_1\cong T_{\psi_1(H)}\mathbb{S}^{n-1}$. Reparametrizing by $t\mapsto 1-t$ and applying again
\eqref{general:initial:conditions}, we obtain that $g$ extends smoothly to $Q_1$ provided that
\begin{align}
\big(f_1(1),f_2(1),f_3(1),f_1'(1),f_2'(1),f_3'(1)\big)
  \;=\; (\zeta_1,0,\zeta_1,0,-\tfrac12,0),\ f_2''(1)=0, \label{initial:values:1:model:C}
\end{align}
where $\zeta_1>0$ is arbitrary.

\begin{proposition}\label{proposition:local:solutions:model:C}
The following statements hold:
\begin{itemize}
    \item[$a)$] The system \eqref{Einstein:CP^n:f_1}-\eqref{Einstein:CP^n:f_3}, subject to the initial conditions \eqref{initial:values:0:model:C}, admits a unique solution defined on $[0,\epsilon_0)$ for some $\epsilon_0>0.$
    \item[$b)$] The system \eqref{Einstein:CP^n:f_1}-\eqref{Einstein:CP^n:f_3}, subject to the final conditions \eqref{initial:values:1:model:C}, admits a unique solution defined on $(1-\epsilon_1,1]$ for some $\epsilon_1>0.$
\end{itemize}
\end{proposition}
\begin{proof} We follow the same argument as in the proof of Proposition \ref{proposition:local:solutions:model:A}. We use the change of variables
\begin{equation}\label{change:variables:0:model:C}
x_1=\frac{f_1^2}{t^2},\ x_2=f_2^2,\ x_3=f_3^2,\
y_1=\frac{f_1f_1'}{t^2}-\frac{f_1^2}{t^3},\ y_2=f_2f_2',\ y_3=f_3f_3'.
\end{equation}
In these variables, the Einstein system \eqref{Einstein:CP^n:f_1}-\eqref{Einstein:CP^n:f_3} takes the form
\begin{equation}\label{EW:form:model:C}
x'=2y,\quad
y'=\frac{1}{t^2}A(x)+\frac{1}{t}B(x,y)+C(t,x,y),
\end{equation}
but now
\begin{align*}
    &A(x)=(A_1,A_2,A_3),\ B(x,y)=(B_1,B_2,B_3),\ C(t,x,y)=(C_1,C_2,C_3),\\[2pt]
    &A_1(x)=\frac{n-1}{4}-\frac{n-1}{8}\Big(\frac{x_2}{x_3}+\frac{x_3}{x_2}\Big),\
      A_2(x)=\frac{x_2^2}{8x_1x_3}-\frac{x_3}{8x_1},\
      A_3(x)=\frac{x_3^2}{8x_1x_2}-\frac{x_2}{8x_1},\\[2pt]
    &B_1(x,y)=-2y_1-(n-1)\Big(\frac{x_1y_2}{x_2}+\frac{x_1y_3}{x_3}\Big),\
      B_2(x,y)=-y_2,\ B_3(x,y)=-y_3,\\[2pt]
    &C_1(t,x,y)=\frac{y_1^2}{x_1}-(n-1)\Big(\frac{y_1y_2}{x_2}+\frac{y_1y_3}{x_3}\Big)-\lambda x_1+\frac{t^2x_1^2}{x_2x_3},\\
    &C_2(t,x,y)=-(n-3)\frac{y_2^2}{x_2}+\frac{n-1}{4}-(n-1)\frac{y_2y_3}{x_3}-\frac{y_1y_2}{x_1}-\lambda x_2-\frac{t^2x_1}{8x_3},\\
    &C_3(t,x,y)=-(n-3)\frac{y_3^2}{x_3}+\frac{n-1}{4}-(n-1)\frac{y_2y_3}{x_2}-\frac{y_1y_3}{x_1}-\lambda x_3-\frac{t^2x_1}{8x_2}.
\end{align*}
The initial conditions \eqref{initial:values:0:model:C} become
\begin{equation}\label{new:initial:conditions:model:C}
(x(0),y(0))=\big((1,\zeta_0^2,\zeta_0^2),(0,0,0)\big).
\end{equation}
As before, at the initial point, we have
\[
A(1,\zeta_0^2,\zeta_0^2)=0,
\quad
2(dA)_{(1,\zeta_0^2,\zeta_0^2)}\cdot(0,0,0)+B(1,\zeta_0^2,\zeta_0^2,0,0,0)=0,
\]
and $A,B,C$ are real-analytic. Hence, by \cite[Section 5]{EW}, there exists a power-series solution defined on $[0,\epsilon_0)$ for some $\epsilon_0>0$.
Uniqueness follows from the invertibility, for $m\ge0$, of
\[
\mathcal{L}_m:=(m+1)\mathrm{Id}-\frac{2}{m+2}\,(dA)_{(1,\zeta_0^{2},\zeta_0^{2})}-(\partial_yB)_{(1,\zeta_0^{2},\zeta_0^{2},0,0,0)},
\]
since in this case 
\[
\det(\mathcal{L}_m)=(m+1)(m+3)^2>0,\ \forall\,m\geq0.\]
This proves $a).$\\

For part $b),$ set $g_i(t):=f_i(1-t)$ and adopt the corresponding variables
\[
x_1=g_1^2,\ x_2=\frac{g_2^2}{t^2},\ x_3=g_3^2,\
y_1=g_1g_1',\ y_2=\frac{g_2g_2'}{t^2}-\frac{g_2^2}{t^3},\ y_3=g_3g_3'.
\]
Arguing analogously to part $a),$ we obtain the existence and uniqueness of solutions on $(1-\epsilon_1,1]$ for some $\epsilon_1>0$.
\end{proof}

\begin{remark}
The previous proposition shows that, once the smoothness conditions at a singular orbit are imposed, the Einstein equations admit a unique real-analytic solution on a neighborhood of that orbit. In other words, the boundary data determined by smoothness fix the power series expansions of $(f_1,f_2,f_3)$ and leave no free parameters. Therefore, a smooth $\mathrm{SO}(n+1)$-invariant Einstein metric on the whole $\mathbb{C}P^n$ would require the two one-sided solutions to agree on $(0,1)$ while satisfying the regularity conditions at both boundary points. As shown in Proposition \ref{non:existence:model:C}, this does not happen.
\end{remark}

\begin{proposition}\label{non:existence:model:C}
    There exists no smooth, globally defined, $\mathrm{SO}(n+1)$-invariant Einstein metric 
    \[
    g=dt^2+\sum\limits_{i=1}^3f_i(t)^2(\cdot,\cdot)\big{|}_{\mathfrak{m}_i\times\mathfrak{m}_i}
    \]
    on $\mathbb{C}P^n$ for which $Q_0$ is totally geodesic.
\end{proposition}
\begin{proof}
Assume that there exists a smooth, globally defined \(\mathrm{SO}(n+1)\)-invariant Einstein metric
\[g=dt^2+\sum\limits_{i=1}^3f_i(t)^2(\cdot,\cdot)\big{|}_{\mathfrak{m}_i\times\mathfrak{m}_i}\]
on \(\mathbb{C}P^n\) for which $Q_0$ is totally geodesic. Then \(g\) necessarily satisfies the boundary conditions \eqref{initial:values:0:model:C} for some \(\zeta_0>0\). Let $A(x)$, $B(x,y)$, and $C(t,x,y)$ be as in the proof of Proposition \ref{proposition:local:solutions:model:C}. Note that, restricted to the plane $\{x_2=x_3,y_2=y_3\}$ we have
\begin{align*}
    &A(x)=(0,0,0),\\
    &B_{1}(x,y)=-2y_1-2(n-1)\frac{x_1y_2}{x_2}=:\tilde{B}_1(x_1,x_2,y_1,y_2)\\
    &B_2(x,y)=B_3(x,y)=-y_2=:\tilde{B}_2(x_1,x_2,y_1,y_2),\\
    &C_1(t,x,y)=\frac{y_1^2}{x_1}-2(n-1)\frac{y_1y_2}{x_2}-\lambda x_1+\frac{t^2x_1^2}{x_2^2}=:\tilde{C}_1(t,x_1,x_2,y_1,y_2),\\
    &C_2(t,x,y)=C_3(t,x,y)=-(2n-4)\frac{y_2^2}{x_2}+\frac{n-1}{4}-\frac{y_1y_2}{x_1}-\lambda x_2-\frac{t^2x_1}{8x_2}=:\tilde{C}_2(t,x_1,x_2,y_1,y_2).
\end{align*}
Consider the two-dimensional initial value problem
\begin{align}\label{two-dimensional:system:model:C}
    \begin{split}
    &(x_1,x_2)'=2(y_1,y_2),\\
    &(y_1,y_2)'=\frac{1}{t}\,\tilde{B}(x_1,x_2,y_1,y_2)+\tilde{C}(t,x_1,x_2,y_1,y_2),\\
    &(x_1(0),x_2(0),y_1(0),y_2(0))=(1,\zeta_0^2,0,0),
    \end{split}
\end{align}
where $\tilde{B}=(\tilde{B}_1,\tilde{B}_2)$ and $\tilde{C}=(\tilde{C}_1,\tilde{C}_2)$. Since $\tilde{B}(1,\zeta_0^2,0,0)=(0,0)$ and
\[
\det\left((m+1)\mathrm{Id}-(\partial_y\tilde{B})_{(1,\zeta_0^2,0,0)}\right)=(m+3)(m+2)>0,\ \forall\,m\ge 0,
\]
it follows from \cite[Section~5]{EW} (see also \cite[Theorem~2.2]{B}) that \eqref{two-dimensional:system:model:C} admits a unique solution, say $(\tilde{x}_1,\tilde{x}_2,\tilde{y}_1,\tilde{y}_2)$. Consequently, $(\tilde{x}_1,\tilde{x}_2,\tilde{x}_2,\tilde{y}_1,\tilde{y}_2,\tilde{y}_2)$ is the unique solution of \eqref{EW:form:model:C}-\eqref{new:initial:conditions:model:C}. Under the change of variables \eqref{change:variables:0:model:C}, the plane $\{x_2=x_3,y_2=y_3\}$ corresponds to $\{f_2=f_3\}.$ Hence, the solution of \eqref{Einstein:CP^n:f_1}-\eqref{Einstein:CP^n:f_3} with initial data \eqref{initial:values:0:model:C} satisfies $f_2=f_3$ on its interval of definition. In particular, it cannot extend smoothly to \(Q_1\), since the smoothness conditions at \(t=1\) require
\[
(f_1(1),f_2(1),f_3(1))=(\zeta_1,0,\zeta_1),
\ \zeta_1>0,
\]
which contradicts \(f_2=f_3\).
\end{proof}

\begin{remark}\label{remark:FS:CaoLabora:model:C}
Proposition~\ref{non:existence:model:C} concerns specifically the singular orbit $Q_0$ and does not exclude the possibility that the other singular orbit $Q_1\simeq\mathbb{RP}^n$ is totally geodesic. For example, a homothetic copy of the Fubini-Study metric is given by
\[
f_1(t)=\frac{2}{\pi}\sin\left(\frac{\pi t}{2}\right),\ f_2(t)=\frac{2}{\pi}\sin\left(\frac{\pi(1-t)}{4}\right),\ f_3(t)=\frac{2}{\pi}\cos\left(\frac{\pi(1-t)}{4}\right).
\]
This gives a smooth globally defined $\mathrm{SO}(n+1)$-invariant Einstein metric. At $t=0$,
\[
f_2'(0)=-\frac{\sqrt{2}}{4},\ f_3'(0)=\frac{\sqrt{2}}{4},
\]
so $Q_0$ is not totally geodesic, whereas $Q_1$ is totally geodesic. For the inhomogeneous Einstein metrics constructed by Cao-Labora and Rodr\'iguez-V\'azquez, the singular orbit $Q_1\simeq\mathbb{RP}^n$ is totally geodesic, whereas $Q_0$ is not.
\end{remark}

\subsection{Model D}\label{Einstein:metrics:D}
Let $\mathfrak{m}_1,\dots,\mathfrak{m}_5$ be the isotropy summands in
\eqref{isotropy:summands:model:D}, and let $(\cdot,\cdot)$ be the $\Ad(\mathrm{U}(5))$-invariant inner product
in \eqref{inner:product:model:D}. An $(\cdot,\cdot)$-orthonormal basis adapted to
$\mathfrak{m}_1\oplus\cdots\oplus\mathfrak{m}_5$ can be chosen as follows:
\begin{align*}
&\begin{aligned}
\mathcal{B}_1=\left\{\frac{\sqrt{-1}}{2}\big(E_{11}+E_{22}-E_{33}-E_{44}\big)\right\},
\end{aligned}\\[0.2em]
&\begin{aligned}
\mathcal{B}_2=&\left\{\frac{E_{13}+E_{24}-E_{31}-E_{42}}{2},\frac{\sqrt{-1}(E_{13}-E_{24}+E_{31}-E_{42})}{2},\right.\\[0.2em]
&\left.\frac{E_{14}-E_{23}+E_{32}-E_{41}}{2},\frac{\sqrt{-1}(E_{14}+E_{23}+E_{32}+E_{41})}{2}\right\},
\end{aligned}\\[0.2em]
&\begin{aligned}
\mathcal{B}_3=&\left\{\frac{E_{13}-E_{24}-E_{31}+E_{42}}{2},\frac{\sqrt{-1}(E_{13}+E_{24}+E_{31}+E_{42})}{2},\right.\\[0.2em]
&\left.\frac{E_{14}+E_{23}-E_{32}-E_{41}}{2},\frac{\sqrt{-1}(E_{14}-E_{23}-E_{32}+E_{41})}{2}\right\},
\end{aligned}\\[0.2em]
&\mathcal{B}_4=\left\{\frac{1}{\sqrt2}(E_{5i}-E_{i5}),\ \frac{\sqrt{-1}}{\sqrt2}(E_{5i}+E_{i5}):i\in\{1,2\}\right\},\\[0.2em]
&\mathcal{B}_5=\left\{\frac{1}{\sqrt2}(E_{5i}-E_{i5}),\ \frac{\sqrt{-1}}{\sqrt2}(E_{5i}+E_{i5}):i\in\{3,4\}\right\}.
\end{align*}
With respect to this basis, the only nonzero structure constants $[ijk]$ (up to permutations of $(i,j,k)$) are
\begin{equation*}\label{constants:ijk:model:D}
[123]=[245]=[345]=4,\ [144]=[155]=1.
\end{equation*}
Moreover, the Killing form $B$ of $\mathfrak{u}(5)$ is
\[
B(X,Y)=10\mathrm{Tr}(XY)-2\mathrm{Tr}(X)\mathrm{Tr}(Y),\ X,Y\in\mathfrak{u}(5).
\]
Therefore,
\[
-B\big|_{\mathfrak{m}_i\times\mathfrak{m}_i}=-10\mathrm{Tr}(XY)
=10(\cdot,\cdot)\big|_{\mathfrak{m}_i\times\mathfrak{m}_i},
\ i=1,2,3,4,5,
\]
and hence $b_i=10$ for $i=1,2,3,4,5$.\\

Consider a $\mathrm{U}(5)$-invariant metric on $\mathbb{C}P^9\setminus\{Q_0,Q_1\}$ of the form
\begin{equation}\label{metric:g:model:D}
g = dt^2+\sum_{i=1}^5 f_i(t)^2\,(\cdot,\cdot)\big|_{\mathfrak{m}_i\times\mathfrak{m}_i}.
\end{equation}
Then the equations \eqref{Einstein2}-\eqref{Einstein3} take the form
\begin{align}
    &\begin{aligned}
        &\frac{f_1''}{f_1}-\frac{4}{f_1^2}-2\left(\frac{f_1^2}{f_2^2f_3^2}-\frac{f_2^2}{f_1^2f_3^2}-\frac{f_3^2}{f_1^2f_2^2}+\frac{1}{8}\frac{f_1^2}{f_4^4}+\frac{1}{8}\frac{f_1^2}{f_5^4}\right)\\[0.5em]
        &+4\left(\frac{f_1'f_2'}{f_1f_2}+\frac{f_1'f_3'}{f_1f_3}+\frac{f_1'f_4'}{f_1f_4}+\frac{f_1'f_5'}{f_1f_5}\right)=-\lambda,
    \end{aligned}\label{Einstein:f1:model:D}\\[0.5em]
    &\begin{aligned}
        &\frac{f_2''}{f_2}-\frac{5}{f_2^2}-\frac{1}{2}\left(\frac{f_2^2}{f_1^2f_3^2}-\frac{f_1^2}{f_2^2f_3^2}-\frac{f_3^2}{f_1^2f_2^2}+\frac{f_2^2}{f_4^2f_5^2}-\frac{f_4^2}{f_2^2f_5^2}-\frac{f_5^2}{f_2^2f_4^2}\right)\\[0.5em]
        &+3\frac{f_2'^2}{f_2^2}+\frac{f_1'f_2'}{f_1f_2}+4\left(\frac{f_2'f_3'}{f_2f_3}+\frac{f_2'f_4'}{f_2f_4}+\frac{f_2'f_5'}{f_2f_5}\right)=-\lambda, \label{Einstein:f2:model:D}
    \end{aligned}\\[0.5em]
    &\begin{aligned}
        &\frac{f_3''}{f_3}-\frac{5}{f_3^2}-\frac{1}{2}\left(\frac{f_3^2}{f_1^2f_2^2}-\frac{f_1^2}{f_2^2f_3^2}-\frac{f_2^2}{f_1^2f_3^2}+\frac{f_3^2}{f_4^2f_5^2}-\frac{f_4^2}{f_3^2f_5^2}-\frac{f_5^2}{f_3^2f_4^2}\right)\\[0.5em]
        &+3\frac{f_3'^2}{f_3^2}+\frac{f_1'f_3'}{f_1f_3}+4\left(\frac{f_2'f_3'}{f_2f_3}+\frac{f_3'f_4'}{f_3f_4}+\frac{f_3'f_5'}{f_3f_5}\right)=-\lambda, \label{Einstein:f3:model:D}
    \end{aligned}\\[0.5em]
    &\begin{aligned}
        &\frac{f_4''}{f_4}-\frac{5}{f_4^2}-\frac{1}{2}\left(\frac{f_4^2}{f_2^2f_5^2}-\frac{f_2^2}{f_4^2f_5^2}-\frac{f_5^2}{f_2^2f_4^2}+\frac{f_4^2}{f_3^2f_5^2}-\frac{f_3^2}{f_4^2f_5^2}-\frac{f_5^2}{f_3^2f_4^2}-\frac{1}{4}\frac{f_1^2}{f_4^4}\right)\\[0.5em]
        &+3\frac{f_4'^2}{f_4^2}+\frac{f_1'f_4'}{f_1f_4}+4\left(\frac{f_2'f_4'}{f_2f_4}+\frac{f_3'f_4'}{f_3f_4}+\frac{f_4'f_5'}{f_4f_5}\right)=-\lambda, \label{Einstein:f4:model:D}
    \end{aligned}\\[0.5em]
    &\begin{aligned}
        &\frac{f_5''}{f_5}-\frac{5}{f_5^2}-\frac{1}{2}\left(\frac{f_5^2}{f_2^2f_4^2}-\frac{f_2^2}{f_4^2f_5^2}-\frac{f_4^2}{f_2^2f_5^2}+\frac{f_5^2}{f_3^2f_4^2}-\frac{f_3^2}{f_4^2f_5^2}-\frac{f_4^2}{f_3^2f_5^2}-\frac{1}{4}\frac{f_1^2}{f_5^4}\right)\\[0.5em]
        &+3\frac{f_5'^2}{f_5^2}+\frac{f_1'f_5'}{f_1f_5}+4\left(\frac{f_2'f_5'}{f_2f_5}+\frac{f_3'f_5'}{f_3f_5}+\frac{f_4'f_5'}{f_4f_5}\right)=-\lambda.\label{Einstein:f5:model:D}
    \end{aligned}
\end{align}

By \eqref{derivatives:psi_0:model:D}, the inner product
\[
4(\cdot,\cdot)\big|_{\mathfrak{m}_1\times\mathfrak{m}_1}+\frac{1}{2}(\cdot,\cdot)\big|_{\mathfrak{m}_5\times\mathfrak{m}_5}
\]
coincides with the standard metric of curvature one on $\mathfrak{m}_1\oplus\mathfrak{m}_5\cong T_{\psi_0(H)}\mathbb{S}^5$.
Therefore, in view of \eqref{general:initial:conditions}, the metric \eqref{metric:g:model:D} extends smoothly to $Q_0$
only if
\begin{align}\label{initial:conditions:model:D}
    \begin{split}
    &(f_1(0),f_2(0),f_3(0),f_4(0),f_5(0))=(0,\zeta_0,\zeta_0,\zeta_0,0),\ f_1''(0)=f_5''(0)=0,\\[0.2em]
    &(f_1'(0),f_2'(0),f_3'(0),f_4'(0),f_5'(0))=\left(2,0,0,0,\frac{1}{\sqrt{2}}\right),
    \end{split}
\end{align}
where $\zeta_0>0$ is arbitrary.\\

On the other hand, \eqref{derivative:psi_1:model:D} implies that $(\cdot,\cdot)$ restricts to the standard metric of curvature one on
$\mathfrak{m}_2\cong T_{\psi_1(H)}\mathbb{S}^4$. Consequently, the corresponding boundary conditions at $t=1$ are
\begin{align}\label{final:conditions:model:D}
    \begin{split}
    &(f_1(1),f_2(1),f_3(1),f_4(1),f_5(1))=(\zeta_1,0,\zeta_1,\xi_1,\xi_1),\ f_2''(1)=0,\\[0.2em]
    &(f_1'(1),f_2'(1),f_3'(1),f_4'(1),f_5'(1))=\left(0,-1,0,0,0\right),
    \end{split}
\end{align}
where $\zeta_1,\xi_1>0$ are arbitrary.

\begin{proposition}\label{local:solutions:model:D}
    The following statements hold:
    \begin{itemize}
        \item[$a)$] The system \eqref{Einstein:f1:model:D}-\eqref{Einstein:f5:model:D}, subject to the initial conditions \eqref{initial:conditions:model:D}, admits infinitely many solutions defined on intervals of the form $[0,\epsilon_0)$ for some $\epsilon_0>0.$
        \item[$b)$] The system \eqref{Einstein:f1:model:D}-\eqref{Einstein:f5:model:D}, subject to the final conditions \eqref{final:conditions:model:D}, admits infinitely many solutions defined on intervals of the form $(1-\epsilon_1,1]$ for some $\epsilon_1>0.$
    \end{itemize}
\end{proposition}\begin{proof}
The proof follows the same steps as the proof of Proposition \ref{local:solutions:model:B}. We only write the argument for \(b)\), since the proof of \(a)\) is similar. Set
\[
g_i(t):=f_i(1-t),\ i=1,\dots,5.
\]
Then \eqref{final:conditions:model:D} becomes
\begin{equation}\label{reparametrized:final:conditions:model:D}
\begin{split}
&(g_1(0),g_2(0),g_3(0),g_4(0),g_5(0))=(\zeta_1,0,\zeta_1,\xi_1,\xi_1),\\
&(g_1'(0),g_2'(0),g_3'(0),g_4'(0),g_5'(0))=(0,1,0,0,0),\ g_2''(0)=0.
\end{split}
\end{equation}
Define
\[
x_1=g_1^2,\ x_2=\frac{g_2^2}{t^2},\ x_3=g_3^2,\ x_4=g_4^2,\ x_5=g_5^2,
\]
\[
y_1=g_1g_1',\ 
y_2=\frac{g_2g_2'}{t^2}-\frac{g_2^2}{t^3},\ 
y_3=g_3g_3',\ 
y_4=g_4g_4',\ 
y_5=g_5g_5'.
\]
With this change of variables, the system \eqref{Einstein:f1:model:D}-\eqref{Einstein:f5:model:D} takes the form
\[
x'=2y,\ 
y'=\frac{1}{t^2}A(x)+\frac{1}{t}B(x,y)+C(t,x,y),
\]
where
\begin{align*}
&A(x)=(A_1(x),A_2(x),A_3(x),A_4(x),A_5(x)),\\
&B(x,y)=(B_1(x,y),B_2(x,y),B_3(x,y),B_4(x,y),B_5(x,y)),\\
&C(t,x,y)=(C_1(t,x,y),C_2(t,x,y),C_3(t,x,y),C_4(t,x,y),C_5(t,x,y)),
\end{align*}
with
\begin{align*}
&\begin{aligned}
&A_1(x)=2\left(\frac{x_1^2}{x_2x_3}-\frac{x_3}{x_2}\right),\ A_2(x)=5-3x_2-\frac{1}{2}\left(\frac{x_1}{x_3}+\frac{x_3}{x_1}+\frac{x_4}{x_5}+\frac{x_5}{x_4}\right),\\[0.4em]
&A_3(x)=\frac{1}{2}\left(\frac{x_3^2}{x_1x_2}-\frac{x_1}{x_2}\right),\ A_4(x)=\frac{1}{2}\left(\frac{x_4^2}{x_2x_5}-\frac{x_5}{x_2}\right),\ A_5(x)=\frac{1}{2}\left(\frac{x_5^2}{x_2x_4}-\frac{x_4}{x_2}\right),
\end{aligned}\\[0.4em]
&\begin{aligned}
&B_1(x,y)=-4y_1,\ B_2(x,y)=-8y_2-\frac{x_2y_1}{x_1}-4\left(\frac{x_2y_3}{x_3}+\frac{x_2y_4}{x_4}+\frac{x_2y_5}{x_5}\right),\\[0.4em]
&B_3(x,y)=-4y_3,\ B_4(x,y)=-4y_4,\ B_5(x,y)=-4y_5,
\end{aligned}\\[0.4em]
&\begin{aligned}
C_1(t,x,y)=&\frac{y_1^2}{x_1}+4+\frac{1}{4}\left(\frac{x_1^2}{x_4^2}+\frac{x_1^2}{x_5^2}\right)-4\left(\frac{y_1y_2}{x_2}+\frac{y_1y_3}{x_3}+\frac{y_1y_4}{x_4}+\frac{y_1y_5}{x_5}\right)\\[0.4em]
&-\lambda x_1-2\frac{x_2}{x_3}t^2,\\[0.4em]
C_2(t,x,y)=&-2\frac{y_2^2}{x_2}-\frac{y_1y_2}{x_1}-4\left(\frac{y_2y_3}{x_3}+\frac{y_2y_4}{x_4}+\frac{y_2y_5}{x_5}\right)-\lambda x_2+\frac{1}{2}\left(\frac{x_2^2}{x_1x_3}+\frac{x_2^2}{x_4x_5}\right),\\[0.4em]
C_3(t,x,y)=&-2\frac{y_3^2}{x_3}+5+\frac{1}{2}\left(\frac{x_3^2}{x_4x_5}
-\frac{x_4}{x_5}
-\frac{x_5}{x_4}\right)-\frac{y_1y_3}{x_1}\\[0.4em]
&-4\left(\frac{y_2y_3}{x_2}+\frac{y_3y_4}{x_4}
+\frac{y_3y_5}{x_5}\right)-\lambda x_3-\frac{1}{2}\frac{x_2}{x_1}t^2,\\[0.4em]
C_4(t,x,y)=&-2\frac{y_4^2}{x_4}+5+\frac{1}{2}\left(\frac{x_4^2}{x_3x_5}-\frac{x_3}{x_5}-\frac{x_5}{x_3}-\frac{1}{4}\frac{x_1}{x_4}\right)-\frac{y_1y_4}{x_1}\\[0.4em]
&-4\left(\frac{y_2y_4}{x_2}
+\frac{y_3y_4}{x_3}
+\frac{y_4y_5}{x_5}\right)-\lambda x_4-\frac{1}{2}\frac{x_2}{x_5}t^2\\[0.4em]
C_5(t,x,y)=&-2\frac{y_5^2}{x_5}+5+\frac{1}{2}\left(\frac{x_5^2}{x_3x_4}-\frac{x_3}{x_4}-\frac{x_4}{x_3}-\frac{1}{4}\frac{x_1}{x_5}\right)-\frac{y_1y_5}{x_1}\\[0.4em]
&-4\left(\frac{y_2y_5}{x_2}
+\frac{y_3y_5}{x_3}
+\frac{y_4y_5}{x_4}\right)-\lambda x_5-\frac{1}{2}\frac{x_2}{x_4}t^2.
\end{aligned}
\end{align*}

The initial conditions are
\[
x(0)=\bigl(\zeta_1^2,1,\zeta_1^2,\xi_1^2,\xi_1^2\bigr),\ 
y(0)=(0,0,0,0,0).
\]
At this point,
\[
A(x(0))=0,\ 
2(dA)_{x(0)}(y(0))+B(x(0),y(0))=0.
\]

Now let \(\mathcal L_m\) and \(D_m\) be defined by \eqref{Lm:definition} and \eqref{Dm:definition}. In the present case,
\[
\det(\mathcal L_m)=
\frac{m(m+1)(m+3)(m+5)^2(m+6)(m+7)(m+8)}{(m+2)^3}.
\]
Therefore \(\mathcal L_m\) is invertible for every \(m\geq 1\), while
\[
\ker(\mathcal L_0)=\mathrm{span}\{(-4,0,1,0,0)\}.
\]

From here the same recursive argument used in Proposition \ref{local:solutions:model:B} applies. Once \(x^0\), \(x^1\), and \(x^2\) are fixed, all higher coefficients are determined uniquely. Here
\[
x^0=x(0)=\bigl(\zeta_1^2,1,\zeta_1^2,\xi_1^2,\xi_1^2\bigr),\ 
x^1=2y(0)=(0,0,0,0,0),
\]
and \(x^2\) must satisfy
\[
\mathcal L_0(x^2)=D_0=((D_0)_1,(D_0)_2,(D_0)_3,(D_0)_4,(D_0)_5),
\]
where
\begin{align*}
&(D_0)_1=(D_0)_3=-2\lambda\zeta_1^2+8+\frac{\zeta_1^4}{\xi_1^4},\\[0.2em]
&(D_0)_2=-2\lambda,\\[0.2em]
&(D_0)_4=(D_0)_5=-2\lambda\xi_1^2+10-\frac{5\zeta_1^2}{4\xi_1^2}.
\end{align*}
Solving this equation gives
\[
\begin{aligned}
x^2
=&\left(
(D_0)_1,
\frac{1}{12}\left((D_0)_2-\frac{(D_0)_1}{\zeta_1^2}-\frac{8(D_0)_4}{5\xi_1^2}\right),
0,
\frac{(D_0)_4}{5},
\frac{(D_0)_4}{5}
\right)\\
&+s(-4,0,1,0,0),\ s\in\mathbb R.
\end{aligned}
\]
Hence the system \eqref{Einstein:f1:model:D}-\eqref{Einstein:f5:model:D}, together with the final conditions \eqref{final:conditions:model:D}, admits infinitely many local solutions near \(t=1\). This proves \(b)\).
\end{proof}
In contrast with Model {\bf B}, in Model {\bf D} the Einstein equation does not produce any relation coming from the non-diagonal components of the Ricci tensor, that is, \eqref{Einstein3} is automatically satisfied for every metric of the form \eqref{metric:g:model:D}. Therefore, the local solutions obtained in Proposition \ref{local:solutions:model:D} are Einstein metrics defined on neighborhoods of the singular orbits. With the notation of the proof above, one can verify that if \(P\) is the map defined in \eqref{map:p}, then
\[
A(Px)=PA(x),\ B(Px,Py)=PB(x,y),\ C(t,Px,Py)=PC(t,x,y),
\]
and also
\[
P(x(0))=x(0),\ P(y(0))=y(0).
\]
Moreover,
\[
P(D_0)=D_0,
\]
and the general solution of \(\mathcal L_0(x^2)=D_0\) satisfies
\[
P(x^2)=x^2.
\]
Hence, the same argument as in Proposition \ref{nonexistence:model:B} applies in this case to obtain the following result.

\begin{proposition}\label{nonexistence:model:D}
There is no smooth, globally defined, \(\mathrm{U}(5)\)-invariant Einstein metric
\[
g=dt^2+\sum_{i=1}^5f_i(t)^2(\cdot,\cdot)\big|_{\mathfrak{m}_i\times\mathfrak{m}_i}
\]
on \(\mathbb{C}P^9\) for which $Q_1$ is totally geodesic.%satisfying \eqref{initial:conditions:model:D} and \eqref{final:conditions:model:D}.
\end{proposition}
\subsection{Model E}\label{Einstein:metrics:E}
Let $\mathfrak{m}_1,\mathfrak{m}_2,\mathfrak{m}_3,\mathfrak{m}_4,\mathfrak{m}_5$ be the spaces described in \eqref{summands:model:E} and let $(\cdot,\cdot)$ be the
$\Ad(\mathrm{Spin}(10))$-invariant inner product on $\mathfrak{spin}(10)\cong\mathfrak{so}(10)$ given by
\eqref{inner:product:model:E}. An $(\cdot,\cdot)$-orthonormal basis adapted to $\mathfrak{m}_1\oplus\mathfrak{m}_2\oplus\mathfrak{m}_3\oplus\mathfrak{m}_4\oplus\mathfrak{m}_5$ is 
\begin{align*}
&\mathcal{B}_1=\left\{\sum_{k=1}^{4}X_{2k,2k-1}\right\},\\[0.2em]
&\mathcal{B}_2=\left\{\sqrt{2}\,(X_{10,2k}+X_{9,2k-1}),\ \sqrt{2}\,(X_{10,2k-1}-X_{9,2k}):k=1,2,3,4\right\},\\[0.2em]
&\mathcal{B}_3=\left\{\sqrt{2}\,(X_{9,2k-1}-X_{10,2k}),\ \sqrt{2}\,(X_{10,2k-1}+X_{9,2k}):k=1,2,3,4\right\},\\[0.2em]
&\mathcal{B}_4=\{Y_1,\dots,Y_6\},\ \mathcal{B}_5=\{Z_1,\dots,Z_6\},
\end{align*}
The nonzero structure constants $[ijk]$ (up to permutations of $(i,j,k)$) are
\begin{equation*}\label{constants:ijk:model:E}
    [122]=[133]=8,\ [145]=24,\ [234]=[235]=48.
\end{equation*}
For $\mathfrak{so}(10)$ one has $B(X,Y)=8\mathrm{Tr}(XY).$ Since $(X,Y)=-\frac{1}{8}\mathrm{Tr}(XY),$ it follows that \(-B=64\,(\cdot,\cdot).\) Therefore,
\[
-B\big|_{\mathfrak{m}_i\times\mathfrak{m}_i}=b_i\,(\cdot,\cdot)\big|_{\mathfrak{m}_i\times\mathfrak{m}_i},\ \textnormal{with}\ b_1=b_2=b_3=b_4=b_5=64.
\]
A $\mathrm{Spin}(10)$-invariant metric on $\mathbb{C}P^{15}\setminus\{Q_0,Q_1\}$ has the form
\begin{equation}\label{metric:g:model:E}
g=dt^2+\sum_{i=1}^{5} f_i(t)^2\,(\cdot,\cdot)\big|_{\mathfrak{m}_i\times\mathfrak{m}_i}.
\end{equation}
The Einstein equations \eqref{Einstein2}-\eqref{Einstein3} become
\begin{align}
&\begin{aligned}&\frac{f_1''}{f_1}-\frac{24}{f_1^2}-2\left(\frac{f_1^2}{f_2^4}+\frac{f_1^2}{f_3^4}+6\frac{f_1^2}{f_4^2f_5^2}-6\frac{f_4^2}{f_1^2f_5^2}-6\frac{f_5^2}{f_1^2f_4^2}\right)\\[0.2em]
&+8\frac{f_1'f_2'}{f_1f_2}+8\frac{f_1'f_3'}{f_1f_3}+6\frac{f_1'f_4'}{f_1f_4}+6\frac{f_1'f_5'}{f_1f_5}=-\lambda,
\end{aligned}\label{Einstein:f1:model:E}\\[0.5em]
&\begin{aligned}
&\frac{f_2''}{f_2}-\frac{32}{f_2^2}-3\left(\frac{f_2^2}{f_3^2f_4^2}-\frac{f_3^2}{f_2^2f_4^2}-\frac{f_4^2}{f_2^2f_3^2}+\frac{f_2^2}{f_3^2f_5^2}-\frac{f_3^2}{f_2^2f_5^2}-\frac{f_5^2}{f_2^2f_3^2}-\frac{1}{6}\frac{f_1^2}{f_2^4}\right)\\[0.2em]
&+7\frac{f_2'^2}{f_2^2}
+\frac{f_1'f_2'}{f_1f_2}
+8\frac{f_2'f_3'}{f_2f_3}
+6\frac{f_2'f_4'}{f_2f_4}
+6\frac{f_2'f_5'}{f_2f_5}
=-\lambda, 
\end{aligned}\label{Einstein:f2:model:E}\\[0.5em]
&\begin{aligned}
&\frac{f_3''}{f_3}-\frac{32}{f_3^2}-3\left(\frac{f_3^2}{f_2^2f_4^2}-\frac{f_2^2}{f_3^2f_4^2}-\frac{f_4^2}{f_3^2f_2^2}+\frac{f_3^2}{f_2^2f_5^2}-\frac{f_2^2}{f_3^2f_5^2}-\frac{f_5^2}{f_3^2f_2^2}-\frac{1}{6}\frac{f_1^2}{f_3^4}\right)\\[0.2em]
&+7\frac{f_3'^2}{f_3^2}
+\frac{f_1'f_3'}{f_1f_3}
+8\frac{f_2'f_3'}{f_2f_3}
+6\frac{f_3'f_4'}{f_3f_4}
+6\frac{f_3'f_5'}{f_3f_5}
=-\lambda,
\end{aligned}\label{Einstein:f3:model:E}\\[0.5em]
&\begin{aligned}
&\frac{f_4''}{f_4}-\frac{32}{f_4^2}-4\left(\frac{f_4^2}{f_2^2f_3^2}-\frac{f_2^2}{f_4^2f_3^2}-\frac{f_3^2}{f_4^2f_2^2}+\frac{1}{2}\frac{f_4^2}{f_1^2f_5^2}-\frac{1}{2}\frac{f_1^2}{f_4^2f_5^2}-\frac{1}{2}\frac{f_5^2}{f_4^2f_1^2}\right)\\[0.2em]
&+5\frac{f_4'^2}{f_4^2}
+\frac{f_1'f_4'}{f_1f_4}
+8\frac{f_2'f_4'}{f_2f_4}
+8\frac{f_3'f_4'}{f_3f_4}
+6\frac{f_4'f_5'}{f_4f_5}
=-\lambda,
\end{aligned}\label{Einstein:f4:model:E}\\[0.5em]
&\begin{aligned}
&\frac{f_5''}{f_5}-\frac{32}{f_5^2}-4\left(
\frac{f_5^2}{f_2^2f_3^2}-\frac{f_2^2}{f_5^2f_3^2}-\frac{f_3^2}{f_5^2f_2^2}
+\frac{1}{2}\frac{f_5^2}{f_1^2f_4^2}-\frac{1}{2}\frac{f_1^2}{f_5^2f_4^2}-\frac{1}{2}\frac{f_4^2}{f_5^2f_1^2}
\right)\\[0.2em]
&+5\frac{f_5'^2}{f_5^2}
+\frac{f_1'f_5'}{f_1f_5}
+8\frac{f_2'f_5'}{f_2f_5}
+8\frac{f_3'f_5'}{f_3f_5}
+6\frac{f_4'f_5'}{f_4f_5}
=-\lambda.
\end{aligned}\label{Einstein:f5:model:E}
\end{align}
By \eqref{derivative:psi0:model:E}, the round metric of curvature one on $K_0/H\approx\mathbb{S}^9$ corresponds, under
$T_{\psi_0(H)}(\mathbb{S}^9)\cong \mathfrak{p}_0=\mathfrak{m}_1\oplus\mathfrak{m}_2$, to
\[
16(\cdot,\cdot)\big|_{\mathfrak{m}_1\times\mathfrak{m}_1}
+2(\cdot,\cdot)\big|_{\mathfrak{m}_2\times\mathfrak{m}_2}.
\]
Thus the smoothness conditions at $t=0$ are
\begin{align}\label{initial:conditions:model:E}
\begin{split}
&(f_1(0),f_2(0),f_3(0),f_4(0),f_5(0))=(0,0,\zeta_0,\zeta_0,\zeta_0),\ \zeta_0>0,\\
&(f_1'(0),f_2'(0),f_3'(0),f_4'(0),f_5'(0))=(4,\sqrt2,0,0,0),\
f_1''(0)=f_2''(0)=0.
\end{split}
\end{align}

Similarly, \eqref{derivative:psi1:model:E} identifies $(\cdot,\cdot)\big|_{\mathfrak{m}_4\times\mathfrak{m}_4}$ with the round metric on
$T_H(K_1/H)\cong T_{\psi_1(H)}\mathbb{S}^6$, hence, the smoothness conditions at $t=1$ are
\begin{align}\label{final:conditions:model:E}
\begin{split}
&(f_1(1),f_2(1),f_3(1),f_4(1),f_5(1))=(\zeta_1,\xi_1,\xi_1,0,\zeta_1),\ \zeta_1,\xi_1>0,\\
&(f_1'(1),f_2'(1),f_3'(1),f_4'(1),f_5'(1))=(0,0,0,-1,0),\
f_4''(1)=0.
\end{split}
\end{align}
\begin{proposition}\label{local:solutions:model:E}
The following statements hold:
\begin{itemize}
    \item[$a)$] The system \eqref{Einstein:f1:model:E}-\eqref{Einstein:f5:model:E}, subject to the initial conditions \eqref{initial:conditions:model:E}, admits infinitely many solutions defined on intervals of the form $[0,\epsilon_0)$ for some $\epsilon_0>0$.
    \item[$b)$] The system \eqref{Einstein:f1:model:E}-\eqref{Einstein:f5:model:E}, subject to the final conditions \eqref{final:conditions:model:E}, admits infinitely many solutions defined on intervals of the form $(1-\epsilon_1,1]$ for some $\epsilon_1>0$.
\end{itemize}
\end{proposition}
\begin{proof}
We argue as in Proposition \ref{local:solutions:model:D}. In this case, we only prove \(a)\), since \(b)\) is obtained in the same way. Set
\[
x_1=\frac{f_1^2}{t^2},\ x_2=\frac{f_2^2}{t^2},\ x_3=f_3^2,\ x_4=f_4^2,\ x_5=f_5^2,
\]
and
\[
y_1=\frac{f_1f_1'}{t^2}-\frac{f_1^2}{t^3},\ 
y_2=\frac{f_2f_2'}{t^2}-\frac{f_2^2}{t^3},\ 
y_3=f_3f_3',\ y_4=f_4f_4',\ y_5=f_5f_5'.
\]
Then the system \eqref{Einstein:f1:model:E}-\eqref{Einstein:f5:model:E} becomes
\[
x'=2y,\ y'=\frac{1}{t^2}A(x)+\frac{1}{t}B(x,y)+C(t,x,y),
\]
where
\begin{align*}
&A(x)=(A_1(x),A_2(x),A_3(x),A_4(x),A_5(x)),\\
&B(x,y)=(B_1(x,y),B_2(x,y),B_3(x,y),B_4(x,y),B_5(x,y)),\\
&C(t,x,y)=(C_1(t,x,y),C_2(t,x,y),C_3(t,x,y),C_4(t,x,y),C_5(t,x,y)),
\end{align*}
with
\begin{align*}
&\begin{aligned}
A_1(x)=&24+2\frac{x_1^2}{x_2^2}-8x_1-12\frac{x_4}{x_5}-12\frac{x_5}{x_4},\\[0.3em]
A_2(x)=&32-8x_2-\frac{1}{2}\frac{x_1}{x_2}-3\frac{x_3}{x_4}-3\frac{x_4}{x_3}-3\frac{x_3}{x_5}-3\frac{x_5}{x_3},\\[0.3em]
A_3(x)=&3\frac{x_3^2}{x_2x_4}+3\frac{x_3^2}{x_2x_5}-3\frac{x_4}{x_2}-3\frac{x_5}{x_2},\\[0.3em]
A_4(x)=&4\frac{x_4^2}{x_2x_3}+2\frac{x_4^2}{x_1x_5}-4\frac{x_3}{x_2}-2\frac{x_5}{x_1},\\[0.3em]
A_5(x)=&4\frac{x_5^2}{x_2x_3}+2\frac{x_5^2}{x_1x_4}-4\frac{x_3}{x_2}-2\frac{x_4}{x_1},
\end{aligned}\\[0.3em]
&\begin{aligned}
B_1(x,y)=&-10y_1-8\frac{x_1y_2}{x_2}-8\frac{x_1y_3}{x_3}-6\frac{x_1y_4}{x_4}-6\frac{x_1y_5}{x_5},\\[0.3em]
B_2(x,y)=&-17y_2-\frac{x_2y_1}{x_1}-8\frac{x_2y_3}{x_3}-6\frac{x_2y_4}{x_4}-6\frac{x_2y_5}{x_5},\\[0.3em]
B_3(x,y)=&-9y_3,\ B_4(x,y)=-9y_4,\ B_5(x,y)=-9y_5,
\end{aligned}\\[0.3em]
&\begin{aligned}
C_1(t,x,y)=&\frac{y_1^2}{x_1}-8\frac{y_1y_2}{x_2}-8\frac{y_1y_3}{x_3}-6\frac{y_1y_4}{x_4}-6\frac{y_1y_5}{x_5}-\lambda x_1+\left(2\frac{x_1^2}{x_3^2}+12\frac{x_1^2}{x_4x_5}\right)t^2
,\\[0.3em]
C_2(t,x,y)=&-6\frac{y_2^2}{x_2}-\frac{y_1y_2}{x_1}-8\frac{y_2y_3}{x_3}-6\frac{y_2y_4}{x_4}-6\frac{y_2y_5}{x_5}-\lambda x_2+\left(3\frac{x_2^2}{x_3x_4}+3\frac{x_2^2}{x_3x_5}\right)t^2
,\\[0.3em]
C_3(t,x,y)=&-6\frac{y_3^2}{x_3}+32-\frac{y_1y_3}{x_1}-8\frac{y_2y_3}{x_2}-6\frac{y_3y_4}{x_4}-6\frac{y_3y_5}{x_5}-\lambda x_3-\left(3\frac{x_2}{x_4}+3\frac{x_2}{x_5}+\frac{1}{2}\frac{x_1}{x_3}\right)t^2
,\\[0.3em]
C_4(t,x,y)=&-4\frac{y_4^2}{x_4}+32-\frac{y_1y_4}{x_1}-8\frac{y_2y_4}{x_2}-8\frac{y_3y_4}{x_3}-6\frac{y_4y_5}{x_5}-\lambda x_4-\left(4\frac{x_2}{x_3}+2\frac{x_1}{x_5}\right)t^2,\\[0.3em]
C_5(t,x,y)=&-4\frac{y_5^2}{x_5}+32-\frac{y_1y_5}{x_1}-8\frac{y_2y_5}{x_2}-8\frac{y_3y_5}{x_3}-6\frac{y_4y_5}{x_4}-\lambda x_5-\left(4\frac{x_2}{x_3}+2\frac{x_1}{x_4}\right)t^2.
\end{aligned}
\end{align*}

The initial conditions give
\[
x(0)=\bigl(16,2,\zeta_0^2,\zeta_0^2,\zeta_0^2\bigr),\ y(0)=(0,0,0,0,0).
\]
Also,
\[
A(x(0))=0,\ 2(dA)_{x(0)}(y(0))+B(x(0),y(0))=0.
\]

Now consider the operators \(\mathcal L_m\) and \(D_m\) defined in \eqref{Lm:definition} and \eqref{Dm:definition}. A direct computation gives
\[
\det(\mathcal L_m)=
\frac{m^2(m+1)(m+3)(m+10)(m+11)(m+12)^2(m+18)}{(m+2)^4}.
\]
So \(\mathcal L_m\) is invertible for every \(m\geq 1\), and
\[
\ker(\mathcal L_0)=
\mathrm{span}\left\{(-64,1,0,0,0),\left(0,0,-\frac32,1,1\right)\right\}.
\]

We can use the same construction as before. Once \(x^0\), \(x^1\), and \(x^2\) are chosen, every higher coefficient is fixed. Here
\[
x^0=x(0)=\bigl(16,2,\zeta_0^2,\zeta_0^2,\zeta_0^2\bigr),\ x^1=2y(0)=(0,0,0,0,0),
\]
and \(x^2\) has to satisfy
\[
\mathcal L_0(x^2)=D_0,
\]
where
\[
D_0=
\bigl(-32\lambda,\,-4\lambda,\,64-2\lambda\zeta_0^2,\,64-2\lambda\zeta_0^2,\,64-2\lambda\zeta_0^2\bigr).
\]
Solving this equation, we get
\[
\begin{aligned}
x^2
=&\left(
\frac{32\lambda\zeta_0^2-2048}{3\zeta_0^2},\ 
0,\ 
16-\frac{\lambda\zeta_0^2}{2},\ 
0,\ 
0
\right)\\
&+s(-64,1,0,0,0)+r\left(0,0,-\frac32,1,1\right),
\ s,r\in\mathbb R.
\end{aligned}
\]
So there are infinitely many choices for \(x^2\), and each one gives a power-series solution. Therefore the system \eqref{Einstein:f1:model:E}-\eqref{Einstein:f5:model:E}, subject to the initial conditions \eqref{initial:conditions:model:E}, has infinitely many local solutions near \(t=0\). This proves \(a)\).
\end{proof}
As in Model {\bf D}, any local solution provided by the proposition above is an Einstein metric defined on a neighborhood of the corresponding singular orbit. Again, if we consider the linear map \(P:\mathbb{R}^5\to\mathbb{R}^5\) given by
\[
P(x_1,x_2,x_3,x_4,x_5)=(x_1,x_2,x_3,x_5,x_4),
\]
one verifies directly from the expressions above that
\[
A(Px)=PA(x),\ B(Px,Py)=PB(x,y),\ C(t,Px,Py)=PC(t,x,y).
\]
In addition,
\[
P(x(0))=x(0),\ P(y(0))=y(0),\ P(D_0)=D_0.
\]
Since every element of \(\ker(\mathcal L_0)\) is fixed by \(P\), every solution of \(\mathcal L_0(x^2)=D_0\) also satisfies
\[
P(x^2)=x^2.
\]
Arguing as in Proposition \ref{nonexistence:model:B}, one can prove that the set \(\{f_4=f_5\}\) is invariant under the system \eqref{Einstein:f1:model:E}-\eqref{Einstein:f5:model:E} subject to \eqref{initial:conditions:model:E}, so any solution locally defined near \(t=0\) cannot be extended smoothly to \([0,1]\), since \(f_4(1)=0<\zeta_1=f_5(1)\). Therefore, we have the following result:

\begin{proposition}\label{nonexistence:model:E}
There is no smooth, globally defined, \(\mathrm{Spin}(10)\)-invariant Einstein metric
\[
g=dt^2+\sum_{i=1}^{5}f_i(t)^2(\cdot,\cdot)\big|_{\mathfrak{m}_i\times\mathfrak{m}_i}
\]
on \(\mathbb{C}P^{15}\) for which $Q_0$ is totally geodesic.%satisfying \eqref{initial:conditions:model:E} and \eqref{final:conditions:model:E}.
\end{proposition}
\section*{Acknowledgements}
We would like to thank Paul Schwahn for his comments and suggestions.
L. Grama is partially supported by FAPESP grants no. 2021/04065-6, and CNPq grant no. 306021/2024-2. Anderson de Araujo was partially supported by FAPEMIG grants no. RED-00133-21 and APQ-04528-22, and CNPQ.

\section*{Ethical Statement}

\paragraph{Data Availability.}
No datasets were generated or analyzed during the current study. Therefore, data sharing is not applicable to this article.
\\

\paragraph{Conflict of Interest.}
The authors declare that they have no conflict of interest.
\\

\paragraph{Research Involving Human Participants and/or Animals.}
This article does not contain any studies involving human participants or animals performed by any of the authors.

\bibliographystyle{apa}

\end{document}